\documentclass[12pt]{amsart}
\usepackage{amssymb}
\usepackage{amscd}
\usepackage{colordvi}
\usepackage{graphicx}
\usepackage{vmargin}
\usepackage{pgf,tikz}
\usepackage{hyperref}
\usepackage{url}
\usepackage{algpseudocode}
\usepackage{color}
\hypersetup{colorlinks   = true,
linkcolor = blue,
urlcolor = black,
     citecolor  = blue}

\usetikzlibrary{arrows}
\setpapersize{USletter}
\setmargrb{1in}{1in}{1in}{1in} 
\newtheorem{theorem}{Theorem}[section]
\newtheorem{proposition}[theorem]{Proposition}

\newtheorem{lemma}[theorem]{Lemma}
\newtheorem{algorithm}[theorem]{Algorithm}
\newtheorem{subroutine}[theorem]{Subroutine}

\newtheorem{defn}[theorem]{Definition}
\newtheorem{question}{Question}

\theoremstyle{definition}

\newtheorem{example}[theorem]{Example}
\newtheorem{remark}[theorem]{Remark}

\newcommand{\R}{\ensuremath{\mathbb{R}}}
\newcommand{\C}{\ensuremath{\mathbb{C}}}
\newcommand{\Z}{\ensuremath{\mathbb{Z}}}
\newcommand{\Q}{\mathbb{Q}}
\newcommand{\mco}{\mathcal{O}}

\newcommand{\pro}{\mathbb{P}}

\newcommand{\mfm}{\mathfrak{m}}

\newcommand{\val}{\text{val}}

\newcommand\be{\begin{equation}}
\newcommand\ee{\end{equation}}
\newcommand\bea{\begin{eqnarray}}
\newcommand\eea{\end{eqnarray}}
\newcommand\bi{\begin{itemize}}
\newcommand\ei{\end{itemize}}
\newcommand\ben{\begin{enumerate}}
\newcommand\een{\end{enumerate}}
\newcommand\ba{\begin{array}}
\newcommand\ea{\end{array}}
\newcommand\bpf{\begin{proof}}
\newcommand\epf{\end{proof}}
\newcommand\bex{\begin{exercise}}
\newcommand\eex{\end{exercise}}

\title{Algorithms for Mumford Curves}

\author{Ralph Morrison and Qingchun Ren}

\begin{document}

\maketitle

\begin{abstract}  Mumford showed that Schottky subgroups of $PGL(2,K)$ give rise to certain curves, now called Mumford curves, over a non-Archimedean field K. Such curves are foundational to subjects dealing with non-Archimedean varieties, including Berkovich theory and tropical geometry. We develop and implement numerical algorithms for Mumford curves over the field of $p$-adic numbers. A crucial and difficult step is finding a good set of generators for a Schottky group, a problem solved in this paper. This result allows us to design and implement algorithms for tasks such as: approximating the period matrices of the Jacobians of Mumford curves; computing the Berkovich skeleta of their analytifications; and approximating points in canonical embeddings. We also discuss specific methods and future work for hyperelliptic Mumford curves.
\end{abstract}

\section{Introduction}

Curves over non-Archimedean fields are of fundamental importance to algebraic geometry and number theory.   Mumford curves are a family of such curves, and are interesting from both a theoretical and computational perspective.  In non-Archimedean geometry, they are quotients of projective space by Schottky groups. In tropical geometry, which looks at the images in $\R^n$ of curves under coordinate-wise valuation, these are balanced graphs with the maximal number of cycles.  For instance, the tropicalization of an elliptic Mumford curve can be realized as a plane cubic in honeycomb form \cite{CS}. 

Let $K$ be an algebraically closed field complete with respect to a nontrivial non-Archimedean valuation. Unless otherwise stated, $|\cdot|$ will denote a choice of norm on $K$ coming from this valuation. 
Let $R=\{x\in{}K|\mathrm{val}(x)\geq{}0\}$ be the valuation ring of $K$. This is a local ring with unique maximal ideal $M=\{x\in{}K|\mathrm{val}(x)>0\}$. Let $k=R/M$ denote the residue field of $K$. 
We are most interested in the field of $p$-adic numbers  $\mathbb{Q}_p$, which unfortunately is not algebraically closed.   
(For this case, $R=\Z_p$, the ring of $p$-adic integers, and $k=\mathbb{F}_p$, the field with $p$ elements.) 
Therefore for theoretical purposes we will often consider $K=\mathbb{C}_p$, the complete algebraic closure of $\mathbb{Q}_p$. (In this case $R$ is much larger, and $k$ is the algebraic closure of $\mathbb{F}_p$.) In most of this paper, choosing elements of $\mathbb{C}_p$ that happen to be elements of $\mathbb{Q}_p$ as inputs for algorithms yields an output once again in $\mathbb{Q}_p$.  This ``$\mathbb{Q}_p$ in, $\mathbb{Q}_p$ out'' property means we may take $K$ to be $\mathbb{Q}_p$ for our algorithmic purposes, while still considering $K=\mathbb{C}_p$ when more convenient for the purposes of theory. 
Much of the theory presented here works for other non-Archimedean fields, such as the field of Puiseux series $\C\{\!\{t\}\!\}$.

We recall some standard definitions and notation for $p$-adic numbers; for further background on the $p$-adics, see \cite{Ho}.  For a prime $p$, the $p$-adic valuation $\val_p:\Q^*\rightarrow \Z$ is defined by $\val_p\left(p^{v}\frac{m}{n}\right)=v$, where $m$ and $n$ are not divisible by $p$.  The usual $p$-adic norm $|\cdot|_p$ on $\Q$ is defined for $a\in\Q^*$ by $|a|_p=\frac{1}{p^{\val_p(a)}}$ and for $0$ by $|0|_p=0$. This means that large powers of $p$ are small in absolute value, and small powers of $p$ are large in absolute value. We will usually omit the subscript $p$ from both $|\cdot|_p$ and $\val_p$.  

The completion of $\mathbb{Q}$ with respect to the $p$-adic norm is denoted $\mathbb{Q}_p$, and is called the field of $p$-adic numbers.  Each nonzero  element $b$ of $\mathbb{Q}_p$ can be written uniquely as
$$b=\sum_{n=v}^\infty a_np^n, $$
where $v\in \Z$,  $a_v\neq 0$ and $a_n\in\{0,1,\ldots,p-1\}$ for all $n$. The $p$-adic valuation and norm extend to this field, and such a sum will have $\val(b)=v$ and $|b|=\frac{1}{p^v}$.  In analog to decimal expansions, we will sometimes write
$$b=\ldots a_Na_{N-1}\ldots a_3a_2a_1a_0.a_{-1}a_{-2}\ldots a_{v}, $$
where the expression trails to the left since higher powers of $p$ are smaller in $p$-adic absolute value. We may approximate $b\in \mathbb{Q}_p$ by a finite sum
$$b\approx\sum_{n=v}^N a_np^n, $$
which will give an error of size at most $\frac{1}{p^{N+1}}$.    

Consider the group $PGL(2,K)$, which acts on $\pro^1(K)$ by treating elements as column vectors.  That is, a matrix acts on the point $(a:b)\in \pro^1(K)$ by acting on the vector $\left(\begin{smallmatrix}a\\b\end{smallmatrix}\right)$ on the left. Viewed on an affine patch, the elements of this group act as fractional linear transformations. We are interested in the action of certain subgroups of  $PGL(2,K)$ called \emph{Schottky groups}, because a Schottky group minimally generated by $g\geq 2$ elements will give rise to a curve of genus $g$.

\begin{defn}{ A $2\times{}2$ matrix is \emph{hyperbolic} if it has two eigenvalues with different valuations.  A \emph{Schottky group} $\Gamma\leq PGL(2,K)$ is a finitely generated subgroup such that every non-identity element is hyperbolic. 
}
\end{defn}

There are many equivalent definitions of Schottky groups, including the following useful characterization.

\begin{proposition} { A subgroup of $PGL(2,K)$ is Schottky if and only if it is free, discrete, and finitely generated.  If the generators are elements of $\mathbb{Q}_p^{2\times2}$, we may replace ``free'' with ``torsion free.''}
\end{proposition}

 Let $\Gamma$  be a Schottky group minimally generated by $\gamma_1,\ldots,\gamma_g$.  The above proposition implies that each element $\gamma{}\in \Gamma$ can be written as a unique shortest product $h_1h_2\dotsb{}h_k$, where each $h_i\in{}\{\gamma{}_1,\dotsc{},\gamma{}_g,\gamma{}_1^{-1},\dotsc{},\gamma{}_g^{-1}\}$. This product is called the \emph{reduced word} for $\gamma$. 

  Let $\Sigma$ be the set of points in $\pro^1(K)$ that are fixed points of elements of $\Gamma$ or limit points of the fixed points.  The group $\Gamma$ acts nicely on $\Omega:=\pro^1(K)\setminus \Sigma$; for this reason we will sometimes refer to $\Sigma$ as \emph{the set of bad points} for $\Gamma$.

\begin{theorem}[Mumford, \cite{Mu1}]{Let $\Gamma=\left<\gamma_1,\ldots,\gamma_g\right>$ and $\Omega$ be as above. Then $\Omega/\Gamma$ is analytically isomorphic to a curve of genus $g$.  We call such a curve a \emph{Mumford curve}.
}
\end{theorem}

In a companion paper to \cite{Mu1} (see \cite{Mu2}), Mumford also considered abelian varieties over non-Archimedean fields.  He showed that these could be represented as $(K^*)^g/Q$, where $Q\in (K^*)^{g\times g}$ is called a \emph{period matrix} for the abelian variety, and represents the multiplicative subgroup generated by its columns.

Since their initial appearance in the 1970s,  a rich theory behind Mumford curves has been developed, largely in the 1980s in such works as \cite{GP}.   However, prior to the work in this paper there have been few numerical algorithms for working with them (an exception being a treatment of hyperelliptic Mumford curves, mostly genus $2$, in  \cite{Ka} from 2007).  We have designed and implemented algorithms that accomplish Mumford curve-based tasks over $\mathbb{Q}_p$ previously absent from the realm of computation, and have made many seemingly theoretical and opaque objects hands-on and tractable.

After  discussing in Section \ref{section:fundamental_domains} a technical hypothesis (``good position'') for the input for our algorithms, we present our main algorithms in Section \ref{section:main_algorithms}. They accomplish the following tasks, where we denote $\Omega/\Gamma$ by $C$:

\bi

\item Given a Schottky group $\Gamma$, find a period matrix $Q$ for the abelian variety $\text{Jac}(C )$  (Algorithm \ref{algperiodmatrix}).

\item Given a Schottky group $\Gamma$, find a triple $(G,\ell,h)$, where 
\bi \item $G$ is a graph,
\item $\ell$ is a length function on $G$ such that the metric graph $(G,\ell)$ is  the abstract tropical curve  which is a skeleton of $C^{an}$ (the analytification of $C$), and
\item $h$ is a  natural equivalence $h:R^g\rightarrow G$ from the rose graph on $g$ petals; 
\ei
this data specifies a point in the tropical Teichm\"{u}ller space described in \cite{CMV} (Algorithm \ref{algorithm:abstract_tropical_curve}).

\item  Given a Schottky group $\Gamma$, find points in a canonical embedding of the curve $C$ into $\pro^{g-1}$ (Algorithm \ref{algcanonical}).

\ei
In Section \ref{goodpositionalgorithmsection}, we present an algorithm to achieve the ``good position'' hypothesis that allows the other algorithms to run efficiently, which in doing so verifies that the input group is Schottky (or proves that the group is not Schottky).  This is the most important result of this paper, as the algorithms in Section \ref{section:main_algorithms} rely heavily upon it.    

We take advantage of a property that makes non-Archimedean valued fields like $\mathbb{Q}_p$ special: $|x+y|\leq{}\max\{|x|,|y|\}$. As a result, the error does not accumulate in the computation. Thus we avoid a dangerous hazard present in doing numerical computation over  $\mathbb{R}$ or $\mathbb{C}$. The computational problems are hard in nature. Efficient computation for similar problems is not common in the literature even for genus $2$ case. Our algorithms are capable of solving genus $2$ and some genus $3$ examples on a laptop in reasonable time (several minutes). However, they are less efficient for larger cases. The reason is that the running time grows exponentially as the requirement on the precision of the output (in terms of the number of digits) grows. One of the future goals is to find a way to reduce the running time for the algorithms.

Other future goals for Mumford curve algorithms (detailed in Section \ref{futurequestionssection}) include natural reversals of the algorithms in Section \ref{section:main_algorithms}.   We are also interested in a particular family of Schottky groups called \emph{Whittaker groups}, defined in Subsection \ref{sectionwhittaker}.   These are the Schottky groups that give rise to hyperelliptic Mumford curves.  Some computations for genus $2$ curves arising from Whittaker groups were done in \cite{Ka}, including computation of Jacobians and finding group representations from ramification points.   Two desirable algorithms in this area include:

\bi

\item  Given a Whittaker group $W$, find an affine equation for $\Omega/W$.

\item  Given a totally split hyperelliptic curve $C$, find a Whittaker group $W$ such that $C\cong \Omega/W$.

\ei
The first can be accomplished if a particular presentation of $W$ is available, and a brute force algorithm in \cite{Ka} can compute the second if the ramification points of $C$ are in a nice position.  Future work removing these requirements and improving efficiency would make hyperelliptic Mumford curves very easy to work with computationally.

\subsection*{Acknowledgements}
We thank our advisor Bernd Sturmfels for guiding us through this project. We also thank Matthew Baker, Melody Chan, Diane Maclagan and Thomas Scanlon for helpful discussions and communications. Both authors were supported by the National Science Foundation through grant DMS-0968882. Ralph Morrison was also supported in part by UC Berkeley, and in part by the Max Planck Institute for Mathematics in Bonn.

\subsection*{Supplementary Material}
We made extensive use of the software package {\tt sage} \cite{Sage}.  Our supplementary files can be found at \url{http://math.berkeley.edu/~ralph42/mumford_curves_supp.html}.  We have also included the files in the arXiv submission of this paper, and they can be obtained by downloading the source. There are minor changes in the {\tt sage} implementation from the description of the algorithms in this paper. The changes are made only for convenience in implementation, and they do not affect the behavior of the algorithms.

\section{Good fundamental domains in $\mathbb{P}^1$ and $(\mathbb{P}^1)^{an}$}\label{section:fundamental_domains}

This section introduces {\it good fundamental domains} and the notion of {\it good position} for generators, both of which will play key roles in our algorithms for Mumford curves. Our main algorithms in Section \ref{section:main_algorithms} require as input Schottky generators in good position, without which the rate of convergence of approximations will drop drastically. For our method of putting generators into good position, see Section \ref{goodpositionalgorithmsection}.

  We start with the usual projective line $\mathbb{P}^1$, then discuss the analytic projective line $(\mathbb{P}^1)^{an}$. Our treatment of good fundamental domains follows Gerritzen and van der Put \cite{GP}. The notion is also discussed by Kadziela \cite{Ka}. The introduction to the analytic projective line follows Baker, Payne and Rabinoff \cite{BPR}.

\begin{defn}
An \emph{open ball} in $\mathbb{P}^1$ is either a usual open ball $B(a,r)=\{x\in{}K:|x-a|<r\}$ or the complement of a usual closed ball $\mathbb{P}^1\backslash{}B(a,r)^+=\{\infty{}\}\cup{}\{x\in{}K:|x-a|>r\}$. A \emph{closed ball} is either a usual closed ball or the complement of a usual open ball.
\end{defn}

The open balls generate a topology on $\mathbb{P}^1$. Both open balls and closed balls are simultaneously open and closed in this topology, as is the case for any non-Archimedean field due to the ultrametric inequality $|x+y|\leq\max\{|x|,|y|\}$.  Let $|K^{\times}|$ denote the image of $K^{\times}$ under $|\cdot|$. If $r\in{}|K^{\times{}}|$, the open ball and the closed ball are distinguished by whether there exist two points $x,y$ in the ball such that $|x-y|$ equals the diameter. The complement of an open ball is a closed ball, and vice versa.

\begin{defn}\label{definition:good_fundamental_domain}
A \emph{good fundamental domain} $F\subset \pro^1$ corresponding to the generators $\gamma{}_1,\dotsc{},\gamma{}_g$ is the complement of $2g$ open balls $B_1,\dotsc{},B_g,B_1',\dotsc{},B_g'$, such that corresponding closed balls $B_1^+,\dotsc{},B_g^+,B_1'^+,\dotsc{},B_g'^+$ are disjoint, and that $\gamma{}_i(\mathbb{P}^1\backslash{}B_i')=B_i^+$ and $\gamma{}_i^{-1}(\mathbb{P}^1\backslash{}B_i)=B_i'^+$ for all $i$. The \emph{interior} of $F$ is $F^{\circ{}}=\mathbb{P}^1\backslash{}(B_1^+\cup{}\dotsb{}\cup{}B_g^+\cup{}B_1'^+\cup{}\dotsb{}\cup{}B_g'^+)$. The \emph{boundary} of $F$ is $F\backslash{}F^{\circ{}}$.
\end{defn}

The definition above implies that $\gamma{}_i(\mathbb{P}^1\backslash{}B_i'^+)=B_i$ and $\gamma{}_i^{-1}(\mathbb{P}^1\backslash{}B_i^+)=B_i'$ for all~$i$.

\begin{example}\label{example:good_fundamental_domain}
(1) Let $K=\mathbb{C}_3$ and $\Gamma{}$ be the group generated by
\begin{equation*}
\gamma{}_1 = \begin{bmatrix} -5 & 32 \\ -8 & 35 \end{bmatrix}, \gamma{}_2 = \begin{bmatrix} -13 & 80 \\ -8 & 43 \end{bmatrix}
\end{equation*}
Both matrices have eigenvalues $27$ and $3$. The matrix $\gamma{}_1$ has left eigenvectors 
$\left(\begin{smallmatrix}1\\1\end{smallmatrix}\right)$ and
 $\left(\begin{smallmatrix}4\\1\end{smallmatrix}\right)$, 
 and $\gamma{}_2$ has left eigenvectors $\left(\begin{smallmatrix}2\\1\end{smallmatrix}\right)$ and $\left(\begin{smallmatrix}5\\1\end{smallmatrix}\right)$. 
 We use the convention that $(z_1:z_2)=z_1/z_2$. Then, $F=\mathbb{P}^1\backslash{}(B_1\cup{}B_1'\cup{}B_2\cup{}B_2')$ where $B_1=B(4,1/9)$, $B_1'=B(1,1/9)$, $B_2=B(5,1/9)$, $B_2'=B(2,1/9)$ is a good fundamental domain relative to the generators $\gamma{}_1$ and $\gamma{}_2$. One can verify as follows. First rewrite
\begin{equation*}
\gamma{}_1z=\frac{-5z+32}{-8z+35}=4+\frac{27(z-1)-81}{-8(z-1)+27}.
\end{equation*}
Suppose that $z\in{}B_1'=B(1,1/9)$. Then, $\mathrm{val}(27(z-1))=3+\mathrm{val}(z-1)\geq{}3+2=5$, and $\mathrm{val}(81)=4$. So $\mathrm{val}(27(z-1)-81)=4$. Also, $\mathrm{val}(-8(z-1)+27)\geq{}\min{}(\mathrm{val}(8(z-1)),\mathrm{val}(27))>\min{}(2,3)=2$. So,
\begin{equation*}
|\gamma{}_1z-4|=\left|\frac{27(z-1)-81}{-8(z-1)+27}\right|>\frac{3^{-4}}{3^{-2}}=1/9.
\end{equation*}
So $\gamma{}_1(B_1')\subset{}\mathbb{P}^1\backslash{}B_1^+$. The other three conditions can be verified similarly.

(2) Let $K=\mathbb{C}_3$ and $\Gamma{}$ be the group generated by
\begin{equation*}
\gamma{}_1 = \begin{bmatrix} -79 & 160 \\ -80 & 161 \end{bmatrix}, \gamma{}_2 = \begin{bmatrix} -319 & 1600 \\ -80 & 401 \end{bmatrix}
\end{equation*}
Both matrices have eigenvalues $81$ and $1$. The matrix $\gamma{}_1$ has left eigenvectors $\left(\begin{smallmatrix}1\\1\end{smallmatrix}\right)$ and $\left(\begin{smallmatrix}2\\1\end{smallmatrix}\right)$, and the matrix $\gamma{}_2$ has left eigenvectors $\left(\begin{smallmatrix}4\\1\end{smallmatrix}\right)$ and $\left(\begin{smallmatrix}5\\1\end{smallmatrix}\right)$. Then, $F=\mathbb{P}^1\backslash{}(B_1\cup{}B_1'\cup{}B_2\cup{}B_2')$ where $B_1=B(2,1/9)$, $B_1'=B(1,1/9)$, $B_2=B(5,1/9)$, $B_2'=B(4,1/9)$ is a good fundamental domain relative to the generators $\gamma{}_1$ and $\gamma{}_2$.

(3) Let $K=\mathbb{C}_3$, and let $\Gamma{}$ be the group generated by
\begin{equation*}
\gamma{}_1 =
\begin{bmatrix} 121 & -120 \\ 40 & -39 \end{bmatrix},
\gamma{}_2 = 
\begin{bmatrix} 121 & -240 \\ 20  & -39 \end{bmatrix},
\gamma{}_3 =
\begin{bmatrix} 401 & -1600 \\ 80 & -319 \end{bmatrix}.
\end{equation*}
All three generators have eigenvalues $1$ and $3^4$. The element $\gamma{}_1$ has eigenvectors $\left(\begin{smallmatrix}1\\1\end{smallmatrix}\right)$ and $\left(\begin{smallmatrix}3\\1\end{smallmatrix}\right)$. The element $\gamma{}_2$ has eigenvectors $\left(\begin{smallmatrix}2\\1\end{smallmatrix}\right)$ and $\left(\begin{smallmatrix}6\\1\end{smallmatrix}\right)$. The element $\gamma{}_3$ has eigenvectors $\left(\begin{smallmatrix}4\\1\end{smallmatrix}\right)$ and $\left(\begin{smallmatrix}5\\1\end{smallmatrix}\right)$. Then, $F=\mathbb{P}^1\backslash{}(B_1\cup{}B_1'\cup{}B_2\cup{}B_2'\cup{}B_3\cup{}B_3')$ where $B_1 = B(1,1/9)$, $ B_1' = B(3,1/9)$, $B_2 = B(2,1/9)$,
$B_2' = B(6,1/9)$,  
$B_3 = B(4,1/9)$, $B_3' = B(5,1/9)$ is a good fundamental domain relative to the generators $\gamma{}_1$, $\gamma{}_2$ and $\gamma{}_3$.
\end{example}

The following lemma follows from Definition \ref{definition:good_fundamental_domain} by induction (see \cite[Theorem 6.2]{Ka}).

\begin{lemma}\label{statement:location_of_gamma_b}
Let $F$ and $\gamma_1,\ldots,\gamma_g$ be as in Definition \ref{definition:good_fundamental_domain}, and let $\gamma\in\Gamma\setminus\{\left(\begin{smallmatrix}1&0\\0&1\end{smallmatrix}\right)\}$ and $b\in \pro^1(K)$.  Write the reduced word for $\gamma{}$ as $h_1h_2\dotsb{}h_k$, where $k\geq{}1$ and $h_i{}\in{}\{\gamma_1^{\pm},{}\ldots,{}\gamma_g^{\pm}\}$ for all $i$.  Assume that $b\notin{}B_j'$ if $h_k=\gamma{}_j$ and $b\notin{}B_j$ if $h_k=\gamma{}_j^{-1}$. Then we have
\begin{equation*}
\gamma{}b\in{}
\begin{cases}
B_i^+, \text{if $h_1=\gamma{}_i$}, \\
B_i'^+, \text{if $h_1=\gamma{}_i^{-1}$}.
\end{cases}
\end{equation*}
\end{lemma}

\bpf  To simplify notation we'll outline the proof for the case where $h_i\in\{\gamma_1,\ldots,\gamma_g\}$ for all $i$, and then describe how to generalize to the case of $h_i\in\{\gamma_1^{\pm},\ldots,\gamma_g^{\pm}\}$.

Write $h_i=\gamma_{a_i}$ for each $i$.  Since $h_k=\gamma_{a_k}$, we know by assumption that $b\notin B'_{a_k}$.  By Definition \ref{definition:good_fundamental_domain} we have $\gamma_{a_k}(\pro^1\setminus B'_{a_k})=B_{a_k}^+$, so $h_kb\in B_{a_k}^{+}$.  By the disjointness of the $2g$ closed balls, we know that $h_kb\notin B_{a_{k-1}}'$, and since $\gamma_{a_{k-1}}(\pro^1\setminus B_{a_{k-1}}')=B_{a_{k-1}}^+$, we have $h_{k-1}h_kb\in B_{a_{k-1}}^+$.  We may continue in this fashion until we find that $h_1h_2\ldots h_kb\in B_{a_1}^{+}$.

The only possible obstruction to the above argument in the case of  $h_i\in\{\gamma_1^{\pm},\ldots,\gamma_g^{\pm}\}$ occurs if $h_i\ldots h_kb\in B_{a_i}'^{+}$ and $h_{i-1}=\gamma_{a_i}$ (or, similarly, if $h_i\ldots h_kb\in B_{a_i}^+$ and $h_{i-1}=\gamma_{a_i}^{-1})$, since the above argument needs $\gamma_{a_i}$ to act on $\pro^1\setminus {B'}_{a_i}^+$.  However, this situation arises precisely when $h_i=\gamma_{a_i}^{-1}=h_{i-1}^{-1}$, meaning that the word is not reduced.  Since we've assumed $h_1\ldots h_k$ is reduced, we have the desired result.
\epf

For a fixed set of generators of $\Gamma{}$, there need not exist a good fundamental domain. If there exists a good fundamental domain for some set of free generators of $\Gamma{}$, we say that the generators are {\it in good position}. Gerritzen and van der Put \cite[\textsection I.4]{GP} proved that there always exists a set of generators in good position. They also proved the following desirable properties for good fundamental domains.

\begin{theorem}\label{statement:good_fundamental_domain}
Let $\Gamma{}$ be a Schottky group, $\Sigma{}$ its set of bad points, and $\Omega{}=\mathbb{P}^1\backslash{}\Sigma{}$.

(1) There exists a good fundamental domain for some set of generators $\gamma_1,\ldots ,\gamma_g$ of $\Gamma$.

Let $F$ be a good fundamental domain for $\gamma_1,\ldots,\gamma_g$, and let $\gamma\in\Gamma$.

(2) If $\gamma{}\neq{}\mathrm{id}$, then $\gamma{}F^{\circ{}}\cap{}F=\phi{}$.

(3) If $\gamma{}\notin{}\{\mathrm{id},\gamma{}_1,\dotsc{},\gamma{}_g,\gamma{}_1^{-1},\dotsc{},\gamma{}_g^{-1}\}$, then $\gamma{}F\cap{}F=\emptyset{}$.

(4) $\cup_{\gamma{}\in{}\Gamma{}} \gamma{}F = \Omega{}$.
\end{theorem}

The statements (2), (3), and (4) imply that $\Omega{}/\Gamma{}$ can be obtained from $F$ by glueing the boundary of $F$. More specifically, $B_i^+\backslash{}B_i$ is glued with $B_i'^+\backslash{}B_i'$ via the action of $\gamma{}_i$. We have designed the following subroutine, which takes any point $p$ in $\Omega{}$ and finds a point $q$ in $F$ such that they are equivalent modulo the action of $\Gamma{}$. This subroutine is useful in developing the algorithms in Section 3 and 4.

\begin{subroutine}[Reducing a point into a good fundamental domain]\label{algorithm:reduction_fundamental_domain}\quad{}
\rm{
\begin{algorithmic}[1]

\Require Matrices $\gamma_1,\ldots \gamma_g$ generating a Schottky group $\Gamma$, a good fundamental domain $F=\mathbb{P}^1\backslash{}(B_1\cup{}\dotsb{}\cup{}B_g\cup{}B_1'\cup{}\dotsb{}\cup{}B_g')$ associated to these generators, and a point $p\in{}\Omega{}$.
\Ensure A point $q\in{}F$ and an element $\gamma{}\in{}\Gamma{}$ such that $q=\gamma{}p$.

\State Let $q\gets{}p$ and $\gamma{}\gets{}\mathrm{id}$.
\While {$p\notin{}F$}
\State If $q\in{}B_i'$, let $q\gets{}\gamma{}_iq$ and $\gamma{}\gets{}\gamma{}_i\gamma{}$.
\State Otherwise, if $q\in{}B_i$, let $q\gets{}\gamma{}_i^{-1}q$ and $\gamma{}\gets{}\gamma{}_i^{-1}\gamma{}$.
\EndWhile
\State \Return $q$ and $\gamma{}$.

\end{algorithmic}}
\end{subroutine}

\begin{proof}
The correctness of this subroutine is clear. It suffices to prove that the algorithm always terminates. Given $p\in{}\Omega{}$, if $p\notin{}F$, by Theorem \ref{statement:good_fundamental_domain}, there exists $\gamma{}^{\circ{}}=h_1h_2\dotsb{}h_k\in{}\Gamma{}$ (where each $h_j$ is $\gamma_i$ or $\gamma_i^{-1}$ for some $i$) such that $\gamma{}^{\circ{}}p\in{}F$. Without loss of generality, we may assume that $\gamma{}^{\circ{}}$ is chosen such that $k$ is the smallest. Steps 3,4 and Lemma \ref{statement:location_of_gamma_b} make sure that we always choose $q\leftarrow{}h_kq$ and $\gamma{}\leftarrow{}h_k\gamma{}$. Therefore, this subroutine terminates with $\gamma{}=\gamma{}^{\circ{}}$.
\end{proof}

We can extend the definition of good fundamental domains to the analytic projective line $(\mathbb{P}^1)^{an}$. In general, analytification of an algebraic variety is defined in terms of multiplicative seminorms. For our special case $(\mathbb{P}^1)^{an}$, there is a simpler description.  As detailed in \cite{Ba}, $(\mathbb{P}^1)^{an}$ consists of four types of points:
\begin{itemize}
\item Type 1 points are just the usual points of $\mathbb{P}^1$.
\item Type 2 points correspond to closed balls $B(a,r)^+$ where $r\in{}|K^{\times{}}|$.
\item Type 3 points correspond to closed balls $B(a,r)^+$ where $r\notin{}|K^{\times{}}|$.
\item Type 4 points correspond to equivalence classes of sequences of nested closed balls $B_1^+\supset{}B_2^+\supset{}\dotsb{}$ such that their intersection is empty.
\end{itemize}

There is a metric on the set of Type 2 and Type 3 points, defined as follows: let $P_1$ and $P_2$ be two such points and let $B(a_1,r_1)^+$ and $B(a_2,r_2)^+$ be the corresponding closed balls.
\begin{itemize}
\item[(1)] If one of them is contained in the other, say $B(a_1,r_1)^+$ is contained in $B(a_2,r_2)^+$, then the distance $d(P_1,P_2)$ is $\log_p(r_2/r_1)$.
\item[(2)] In general, there is a unique smallest closed ball $B(a_3,r_3)^+$ containing both of them. Let $P_3$ be the corresponding point. Then, $d(P_1,P_2)$ is defined to be $d(P_1,P_3)+d(P_3,P_2)$.
\end{itemize}
The metric can be extended to Type 4 points.

This metric makes $(\mathbb{P}^1)^{an}$ a tree with infinite branching, as we now describe. There is a unique path connecting any two points $P_1$ and $P_2$. In case (1) above, the path is defined by the isometry $t\mapsto{}B(a_1,p^t)^+$, $t\in{}[\log{}(r_1),\log{}(r_2)]$. It is straightforward to check that $B(a_1,r_2)^+=B(a_2,r_2)^+$. In case (2) above, the path is the concatenation of the paths from $P_1$ to $P_3$ and from $P_3$ to $P_2$. Then, Type 1 points become limits of Type 2 and Type 3 points with respect to this metric. More precisely, if $x\neq{}\infty{}$, then it lies at the limit of the path $t\mapsto{}B(x,p^{-t})^+$, $t\in{}[0,+\infty{})$. Type 1 points behave like leaves of the tree at infinity. For any two Type 1 points $x,y$, there is a unique path in $(\mathbb{P}^1)^{an}$ connecting them, which has infinite length.

\begin{defn}
Let $\Sigma{}$ be a discrete subset in $\mathbb{P}^1$. The subtree of $(\mathbb{P}^1)^{an}$ \emph{spanned by} $\Sigma{}$, denoted $T(\Sigma{})$, is the union of all paths connecting all pairs of points in $\Sigma{}$.
\end{defn}

An {\it analytic open ball} $B(a,r)^{an}$ is a subset of $(\mathbb{P}^1)^{an}$ whose set of Type 1 points is just $B(a,r)$ and whose Type 2, 3, and 4 points correspond to closed balls $B(a',r')^+\subset{}B(a,r)$ and the limit of sequences of such closed balls. An {\it analytic closed ball} is similar, with $B(a,r)$ replaced with $B(a,r)^+$. Just as in the case of balls in $\mathbb{P}^1$, the analytic closed ball $(B^+)^{an}$ is not the closure of $B^{an}$ in the metric topology of $(\mathbb{P}^1)^{an}$. The complement of an analytic open ball is an analytic closed ball, and vice versa. In an analytic closed ball $(B(a,r)^+)^{an}$ such that $r\in{}|K^\times{}|$, the {\it Gaussian point} is the Type 2 point corresponding to $B(a,r)^+$. An {\it analytic annulus} is $B\backslash{}B'$, where $B$ and $B'$ are analytic balls such that $B'\subsetneqq{}B$. If $B$ is an analytic open (resp. closed) ball and $B'$ is an analytic closed (resp. open) ball, then $B\backslash{}B'$ is an {\it analytic open annulus} (resp. {\it analytic closed annulus}). A special case of analytic open annulus is the complement of a point in an analytic open ball.

Any element of $PGL(2,K)$ sends open balls to open balls and closed balls to closed balls. Thus, there is a well defined action of $PGL(2,K)$ on~$(\mathbb{P}^1)^{an}$.

\begin{defn}\label{definition:analytic_good_fundamental_domain}
A \emph{good fundamental domain} $F\subset (\mathbb{P}^1)^{an}$ corresponding to the generators $\gamma{}_1,\dotsc{},\gamma{}_g$ is the complement of $2g$ analytic open balls $B_1^{an},\dotsc{},B_g^{an},B_1'^{an},\dotsc{},B_g'^{an}$, such that the corresponding analytic closed balls $(B_1^+)^{an},\dotsc{},(B_g^+)^{an},(B_1'^+)^{an},\dotsc{},(B_g'^+)^{an}$ are disjoint, and that $\gamma{}_i((\mathbb{P}^1)^{an}\backslash{}B_i'^{an})=(B_i^+)^{an}$ and $\gamma{}_i^{-1}((\mathbb{P}^1)^{an}\backslash{}B_i^{an})=(B_i'^+)^{an}$. The \emph{interior of $F$} is $F^{\circ{}}=(\mathbb{P}^1)^{an}\backslash{}((B_1^+)^{an}\cup{}\dotsc{}\cup{}(B_g^+)^{an}\cup{}(B_1'^+)^{an}\cup{}\dotsc{}\cup{}(B_g'^+)^{an})$. The \emph{boundary of $F$} is $F\backslash{}F^{\circ{}}$.
\end{defn}

Definition \ref{definition:analytic_good_fundamental_domain} implies that $\gamma{}_i((\mathbb{P}^1)^{an}\backslash{}(B_i'^+)^{an})=B_i^{an}$ and $\gamma{}_i^{-1}((\mathbb{P}^1)^{an}\backslash{}(B_i^+)^{an})=B_i'^{an}$.

We now argue that there is a one-to-one correspondence between good fundamental domains in $\pro^1$ and good fundamental domains in $(\pro^1)^{an}$.  (This fact is well-known, though seldom explicitly stated in the literature; for instance, it's taken for granted in the later chapters of \cite{GP}.)  If $\mathbb{P}^1\backslash{}(B_1\cup{}\dotsb{}\cup{}B_g\cup{}B_1'\cup{}\dotsb{}\cup{}B_g')$ is a good fundamental domain in $\mathbb{P}^1$, then $(\mathbb{P}^1)^{an}\backslash{}(B_1^{an}\cup{}\dotsb{}\cup{}B_g^{an}\cup{}B_1'^{an}\cup{}\dotsb{}\cup{}B_g'^{an})$ is a good fundamental domain in $(\mathbb{P}^1)^{an}$. Indeed, since the closed balls $B_1^+,\dotsc{},B_g^+,B_1'^+,\dotsc{},B_g'^+$ are disjoint, and the corresponding analytic closed balls consist of points corresponding to closed balls contained in $B_1^+,\dotsc{},B_g^+,B_1'^+,\dotsc{},B_g'^+$ and their limits, the analytic closed balls are also disjoint. Conversely, if $(\mathbb{P}^1)^{an}\backslash{}(B_1^{an}\cup{}\dotsb{}\cup{}B_g^{an}\cup{}B_1'^{an}\cup{}\dotsb{}\cup{}B_g'^{an})$ is a good fundamental domain in $(\mathbb{P}^1)^{an}$, then $\mathbb{P}^1\backslash{}(B_1\cup{}\dotsb{}\cup{}B_g\cup{}B_1'\cup{}\dotsb{}\cup{}B_g')$ is a good fundamental domain in $\mathbb{P}^1$, because the classical statement can be obtained from the analytic statement by considering only Type 1 points. This correspondence allows us to abuse notation by not distinguishing the classical case and the analytic case. Theorem \ref{statement:good_fundamental_domain} is also true for analytic good fundamental domains.

Another analytic object of interest to us is the minimal skeleton of the analytification of the genus $g$ curve $\Omega/\Gamma{}$ (a task that is part of Algorithm \ref{algorithm:abstract_tropical_curve}), so we close this section with background information on this object. The following definitions are taken from Baker, Payne and Rabinoff \cite{BPR}, with appropriate simplification.

\begin{defn}
\begin{itemize}
\item[(1)] The \emph{skeleton} of an open annulus $B\backslash{}B'$ is the straight path between the Gaussian point of $B$ and the Gaussian point of $B'$.
\item[(2)] Let $C$ be a smooth curve over $K$. A \emph{semistable vertex set} $V$ is a finite set of Type 2 points in $C^{an}$ such that $C^{an}\backslash{}V$ is the disjoint union of open balls and open annuli. The \emph{skeleton corresponding to $V$} is the union of $V$ with all skeleta of these open annuli.
\item[(3)] If $\mathrm{genus}(C)\geq{}2$, then $C^{an}$ has a unique \emph{minimal skeleton}. The minimal skeleton is the intersection of all skeleta. If $\mathrm{genus}(C)\geq{}2$ and $C$ is complete, then the minimal skeleton is a finite metric graph.  We sometimes call this minimal skeleton the \emph{abstract tropical curve} of $C^{an}$.
\end{itemize}
\end{defn}

\begin{defn}
An \emph{algebraic semistable model} of a smooth curve $C$ over $K$ is a scheme $X$ over $R$ whose generic fiber $X_K$ is isomorphic to $C$ and whose special fiber $X_k$ satisfies
\begin{itemize}
\item $X_k$ is a connected and reduced curve, and
\item all singularities of $X_k$ are ordinary double points.
\end{itemize}
\end{defn}

Work towards algorithmic computation of semistable models is discussed in such works as \cite[\textsection1.2]{AW} and \cite[\textsection3.1]{BW}, though such computation is in general a hard problem.

Semistable models are related to skeleta in the following way: take a semistable model $X$ of $C$. Associate a vertex for each irreducible component of $X_k$. For each ordinary intersection of two irreducible components in $X_k$, connect an edge between the two corresponding vertices. The resulting graph is combinatorially a skeleton of $C^{an}$.

\section{Algorithms Starting With a Schottky Group}\label{section:main_algorithms}

If we have a Schottky group $\Gamma=\left<\gamma_1,\ldots,\gamma_g\right>$ in terms of its generators, there are many objects we wish to compute for the corresponding curve $\Omega/\Gamma$, such as the Jacobian of the curve, the minimal skeleton of the analytification of the curve, and a canonical embedding for the curve.  In this section we present algorithms for numerically computing these three objects, given the input of a Schottky group with generators in good position.  For an algorithm that puts arbitrary generators of a Schottky group into good position, see Section~\ref{goodpositionalgorithmsection}.

\begin{remark}  Several results in this section are concerned with the accuracy of numerical approximations. Most of our results will be of the form
$$\left|\frac{\text{estimate}}{\text{actual}}-1\right|=\text{size of error term $\leq$ a small real number of the form $p^{-N}$},$$
where we think of $N\gg 0$.  This is equivalent to
$$\frac{\text{estimate}}{\text{actual}}-1=\text{error term $=$ a $p$-adic number of the form $bp^N$},$$
 where $|b|\leq 1$.  So, since $|p^N|=p^{-N}$, the \emph{size} of the error term is a \emph{small} power of $p$, while the \emph{error term itself} is a \emph{large} power of $p$ (possibly with a constant that doesn't matter much). 
 
Rearranging the second equation gives
$$\text{estimate}=\text{actual}+\text{actual}\cdot bp^N, $$
meaning that we are considering not the \emph{absolute} precision of our estimate, but rather the \emph{relative} precision.  In this case we would say that our estimate is of \emph{relative precision $O(p^N)$}.  So if we desire relative precision $O(p^N)$, we want the \emph{actual} error term to be $p^{N}$ (possibly with a constant term with nonnegative valuation), and the \emph{size} of the error term to be at most $p^{-N}$.
\end{remark}

\subsection{The Period Matrix of the Jacobian}

Given a Schottky group $\Gamma=\left<\gamma_1,\ldots,\gamma_g\right>$, we wish to find a period matrix $Q$ so that $\text{Jac}(\Omega/\Gamma)\cong (K^*)^g/Q$.  First we'll set some notation.  For any parameters $a,b\in \Omega$, we introduce the following analytic function in the unknown $z$, called a \emph{theta function}:
\begin{equation*}
\Theta{}(a,b;z): = \prod_{\gamma{}\in{}\Gamma{}}\frac{z-\gamma{}a}{z-\gamma{}b}.
\end{equation*}
Note that if $\Gamma$ is defined over $\mathbb{Q}_p$  and $a,b,z\in\mathbb{Q}_p$, then $\Theta(a,b;z)\in \mathbb{Q}_p\cup\{\infty\}$.  (This is an instance of ``$\mathbb{Q}_p$ in, $\mathbb{Q}_p$ out.'')
For any $\alpha\in \Gamma$ and $a\in\Omega$, we can specialize to 
\begin{equation*}
u_{\alpha{}}(z) := \Theta{}(a,\alpha{}a;z).
\end{equation*}
 It is shown in \cite[II.3]{GP} that the function $u_{\alpha}(z)$ is in fact independent of the choice of $a$. This is because for any choice of $a,b\in \Omega$ we have
\begin{align*}
\frac{\Theta{}(a,\alpha{}a;z)}{\Theta{}(b,\alpha{}b;z)}=&\prod_{\gamma\in\Gamma}\left(\frac{z-\gamma{}a}{z-\gamma{}\alpha a}\frac{z-\gamma{}\alpha b}{z-\gamma{}b}\right)
=\prod_{\gamma\in\Gamma}\left(\frac{z-\gamma{}a}{z-\gamma{}b}\frac{z-\gamma{}\alpha b}{z-\gamma\alpha{}a}\right)
\\=&\prod_{\gamma\in\Gamma}\frac{z-\gamma{}a}{z-\gamma{}b}\cdot\prod_{\gamma\in \Gamma}\frac{z-\gamma{}\alpha b}{z-\gamma\alpha{}a}
=\prod_{\gamma\in\Gamma}\frac{z-\gamma{}a}{z-\gamma{}b}\cdot\prod_{\gamma\in \Gamma}\frac{z-\gamma{} b}{z-\gamma{}a}
\\=&\Theta(a,b;z)\cdot\Theta(b,a;z)=1.
\end{align*}

From \cite[VI.2]{GP} we have a formula for the period matrix $Q$ of $\text{Jac}(\Omega/\Gamma)$:

\begin{theorem}\label{theorem:computing_period_matrix}
The period matrix $Q$ for $\text{Jac}(\Omega/\Gamma)$ is given by
\begin{equation*}
Q_{ij} = \frac{u_{\gamma{}_i}(z)}{u_{\gamma{}_i}(\gamma{}_j{}z)},
\end{equation*}
where $z$ is any point in $\Omega$.
\end{theorem}

As shown in \cite[II.3]{GP}, the choice of $z$ does not affect the value of $Q_{ij}$.

Theorem \ref{theorem:computing_period_matrix} implies that in order to compute each $Q_{ij}$, it suffices to find a way to compute $\Theta{}(a,b;z)$. Since a theta function is defined as a product indexed by the infinite group $\Gamma$, approximation will be necessary.  Recall that each element $\gamma{}$ in the free group generated by $\gamma{}_1,\dotsc{},\gamma{}_g$ can be written in a unique shortest product $h_1h_2\dotsb{}h_k$ called the reduced word, where each $h_i\in{}\{\gamma{}_1,\dotsc{},\gamma{}_g,\gamma{}_1^{-1},\dotsc{},\gamma{}_g^{-1}\}$.  We can approximate $\Theta{}(a,b;z)$ by replacing the product over $\Gamma{}$ with a product over $\Gamma{}_m$, the set of elements of $\Gamma{}$ whose reduced words have length $\leq{}m$. More precisely, we approximate $\Theta{}(a,b;z)$ with
\begin{equation*}
\Theta{}_m(a,b;z) := \prod_{\gamma{}\in{}\Gamma{}_m}\frac{z-\gamma{}a}{z-\gamma{}b},
\end{equation*}
where
\begin{equation*}
\Gamma_m=\{h_1h_2\ldots h_k\,|\, 0\leq k\leq m, h_i\in \{\gamma_1^{\pm},\ldots \gamma_g^{\pm}\}, h_i\neq h_{i+1}^{-1} \text{ for any i}\}.
\end{equation*}
With this approximation method, we are ready to describe an algorithm for computing $Q$.

\begin{algorithm}[Period Matrix Approximation]\label{algperiodmatrix}\quad{}
{\rm
\begin{algorithmic}[1]

\Require Matrices $\gamma_1,\ldots \gamma_g\in \mathbb{Q}_p^{2\times 2}$ generating a Schottky group $\Gamma$ in good position, and an integer $n$ to specify desired relative precision.
\Ensure An approximation for a period matrix $Q$ for $\text{Jac}(\Omega/\Gamma)$ up to relative precision~$O(p^n)$.

\State Choose suitable $p$-adic numbers $a$ and $z$ as described in Theorem \ref{theorem:approximation_theorem}.
\State Based on $n$, choose a suitable positive integer $m$ as described in Remark \ref{remark:choosing_m}.
\For {$1\leq i,j\leq g$}
\State Compute
$Q_{ij}=\Theta_m(a,\gamma_i(a);z)/\Theta_m(a,\gamma_i(a);\gamma_j(z))$
\EndFor
\State \Return $Q$.

\end{algorithmic}
}
\end{algorithm}

The complexity of this algorithm is in the order of the number of elements in $\Gamma{}_m$, which is exponential in $m$. The next issue is that to achieve certain precision in the final result, we need to know how large $m$ needs to be. Given a good fundamental domain $F$ for the generators $\gamma{}_1,\dotsc{},\gamma{}_g$, we are able to give an upper bound on the error in our estimation of $\Theta{}$ by $\Theta{}_m$.  (Algorithm \ref{algperiodmatrix} would work even if the given generators were not in good position, but would in general require a very large $m$ to give the desired convergence. See Example \ref{example:period_matrix}(4).)

To analyze the convergence of the infinite product
\begin{equation*}
\Theta{}(a,\gamma_i(a);z) = \prod_{\gamma{}\in{}\Gamma{}}\frac{z-\gamma{}a}{z-\gamma{}\gamma{}_ia},
\end{equation*}
we need to know where $\gamma{}a$ and $\gamma{}\gamma{}_ia$ lie.   We can determine this by taking the metric of $(\mathbb{P}^1)^{an}$ into consideration. Assume that $\infty{}$ lies in the interior of $F$. Let $S=\{P_1,\dotsc{},P_g,P_1',\dotsc{},P_g'\}$ be the set of points corresponding to the set of closed balls $\{B_1^+,\dotsc{},B_g^+,B_1'^+,\dotsc{},B_g'^+\}$ from the characterization of the good fundamental domain. Let $c$ be the smallest pairwise distance between these points.  This distance $c$ will be key for determining our choice of $m$ in the algorithm.

\begin{proposition}\label{statement:distance_gamma_a}
Let $F$, $S$, and $c$ be as above.  Suppose the reduced word for $\gamma{}$ is $h_1h_2\dotsb{}h_k$, where $k\geq{}0$. Then $d(\gamma{}P_i,S)\geq{}kc$ for all $i$ unless $h_k=\gamma{}_i^{-1}$, and $d(\gamma{}P_i',S)\geq{}kc$ unless $h_k=\gamma{}_i$.
\end{proposition}

\begin{proof}
We will prove this proposition by induction. If $k=0$, there is nothing to prove. Let $k>0$, and assume that the claim holds for all integers $n$ with $0\leq n<k$. Without loss of generality, we may assume $h_1=\gamma{}_1$. Let $B^+$ be the closed disk corresponding to $P_i$. By Lemma \ref{statement:location_of_gamma_b}, we have $\gamma{}(B^+)\subset{}B_1$. This means $P_1$ lies on the unique path from $\gamma{}P_i$ to $\infty{}$. Since we assumed $\infty{}\in{}F$, $p_1$ lies on the unique path from $\gamma{}P_i$ to any point in $S$. Thus, 
\begin{align*}
d(\gamma{}P_i,S) &= d(\gamma{}P_i,P_1) \\
&= d(\gamma{}_1^{-1}\gamma{}P_i,\gamma{}_1^{-1}P_1) \\
&= d(h_2h_3\dotsb{}h_kP_i,P_1').
\end{align*}
Let $P=P_j$ if $h_2=\gamma{}_j$ and $P=P_j'$ if $h_2=\gamma{}_j^{-1}$. By the same argument as above, $P$ lies on the unique path from $h_2h_3\dotsb{}h_kP_i$ to $P_1'$. The reducedness of the word $h_1h_2\dotsc{}h_k$ guarantees that $P\neq{}P_1'$. So
\begin{align*}
d(\gamma{}P_i,S) &= d(h_2h_3\dotsb{}h_kP_i,P_1') \\
&= d(h_2h_3\dotsb{}h_kP_i,P)+d(P,P_1') \\
&\geq{} (k-1)c+c=kc.
\end{align*}
The last step follows from the inductive hypothesis. The proof of the second part of this proposition is similar.
\end{proof}

\begin{proposition}\label{statement:approximation_fraction}
Let $F$, $S$, and $c$ be as above.  Let $z\in{}F$ and $a\in{}B_i'^+\backslash{}B_i'$ such that $a$, $z$, and $\infty{}$ are distinct modulo the action of $\Gamma{}$.   Suppose the reduced word for $\gamma{}\in \Gamma$ is $h_1h_2\dotsb{}h_k$.  If $k\geq{}2$ and $h_k\neq{}\gamma{}_i^{-1}$, then
\begin{equation*}
\left|\frac{z-\gamma{}a}{z-\gamma{}\gamma{}_ia}-1\right| \leq{} p^{-c(k-1)}.
\end{equation*}
\end{proposition}

\begin{proof}
Our choice of $a$ guarantees that both $a$ and $\gamma{}_ia$ are in $F$. Without loss of generality, we may assume that $h_k=\gamma{}_1$. Then, both $h_ka$ and $h_k\gamma{}_ia$ are in $B_1^+$. So both $\gamma{}a$ and $\gamma{}\gamma{}_ia$ lie in $h_1h_2\dotsb{}h_{k-1}B_1^+$, which is contained in some $B=B_j$ or $B_j'$. By Proposition \ref{statement:distance_gamma_a}, the points in $(\mathbb{P}^1)^{an}$ corresponding to the disks $h_1h_2\dotsb{}h_{k-1}B_1^+$ and $B^+$ have distance at least $c(k-1)$. This implies $\mathrm{diam}(h_1h_2\dotsb{}h_{k-1}B_1^+)\leq{}p^{-c(k-1)}\mathrm{diam}(B^+)$. Therefore, $|\gamma{}a-\gamma{}\gamma{}_ia|\leq{}p^{-c(k-1)}\mathrm{diam}(B^+)$. On the other hand, since $z\notin{}B$ and $\gamma{}\gamma{}_ia\in{}B$, we have $|z-\gamma{}\gamma{}_ia|\geq{}\mathrm{diam}(B^+)$.   This means that
$$\left|\frac{z-\gamma{}a}{z-\gamma{}\gamma{}_ia}-1\right|=\left|\frac{\gamma{}\gamma{}_ia-\gamma{}a}{z-\gamma{}\gamma{}_ia}\right|\leq\frac{p^{-c(k-1)}\text{diam}(B^+)}{\text{diam}(B^+)}=p^{-c(k-1)},$$
as claimed.
\end{proof}

We are now ready to prove our approximation theorem, which is a new result that allows one to determine the accuracy of an approximation of a ratio of theta functions.  It is similar in spirit to \cite[Theorem 6.10]{Ka}, which is an approximation result for a particular subclass of Schottky groups called Whittaker groups (see Subsection \ref{sectionwhittaker} of this paper for more details).  Our result is more general, as there are many Schottky groups that are not Whittaker.

\begin{theorem}\label{theorem:approximation_theorem}
Suppose that the given generators $\gamma{}_1,\dotsc{},\gamma{}_g$ of $\Gamma$ are in good position, with corresponding good fundamental domain $F$ and disks $B_1,\dotsc{},B_g,B_1',\dotsc{},B_g'$. Let $m\geq{}1$. In Algorithm \ref{algperiodmatrix}, if we choose $a\in{}B_i'^+\backslash{}B_i'$ and $z\in{}B_j'^+\backslash{}B_j'$ such that $a\neq{}z$, then
\begin{equation*}
\left|\frac{\Theta_m(a,\gamma_i(a);z)/\Theta_m(a,\gamma_i(a);\gamma_j(z))}{\Theta{}(a,\gamma_i(a);z)/\Theta{}(a,\gamma_i(a);\gamma_j(z))}-1\right| \leq{} p^{-cm},
\end{equation*}
where $c$ is the constant defined above.
\end{theorem}

\begin{proof}
Our choice of $z$ guarantees that both $z$ and $\gamma{}_jz$ are in $F$. Thus, if $\infty{}$ lies in the interior of $F$, then this theorem follows directly from Proposition \ref{statement:approximation_fraction}. The last obstacle is to remove the assumption on $\infty{}$. We observe that $Q_{ij}$ is a product of cross ratios:
\begin{equation*}
\frac{\Theta{}(a,\gamma_ia;z)}{\Theta(a,\gamma_ia;\gamma_jz)} = \prod_{\gamma{}\in{}\Gamma{}}\frac{(z-\gamma{}a)(\gamma{}_jz-\gamma{}\gamma{}_ia)}{(z-\gamma{}\gamma{}_ia)(\gamma{}_jz-\gamma{}a)}.
\end{equation*}
Therefore, each term is invariant under any projective automorphism of $\mathbb{P}^1$. Under such an automorphism, any point in the interior of $F$ can be sent to $\infty{}$.
\end{proof}

As a special case of this approximation theorem, suppose that we want to compute the period matrix for the tropical Jacobian of $C$, which is the matrix $(\mathrm{val}(Q_{ij}))_{g\times{}g}$. We need only to compute $Q_{ij}$ up to relative precision $O(1)$. Thus, setting $m=0$ suffices. In this case, each of the products $\Theta_m(a,\gamma_i(a);z)$, $\Theta_m(a,\gamma_i(a);\gamma_j(z))$ has only one term.

\begin{remark}\label{remark:choosing_m}  If we wish to use Algorithm \ref{algperiodmatrix} to compute a period matrix $Q$ with relative precision $O(p^n)$ (meaning that we want $p^{-cm}\leq p^{-n}$ in Theorem \ref{theorem:approximation_theorem}), we must first compute $c$.  As above, $c$ is defined to be the minimum distance between pairs of the points $P_1,\ldots ,P_g,P_1',\ldots,P_g'\in (\pro^{1})^{an}$ corresponding to the balls $B_1,\ldots, B_g,B_1',\ldots,B_g'$ that characterize our good fundamental domain.  Once we have computed $c$ (perhaps by finding a good fundamental domain using the methods of Section \ref{goodpositionalgorithmsection}), then by Theorem \ref{theorem:approximation_theorem}  we must choose $m$ such that $cm\geq n$, so $m=\lceil n/c\rceil$ will suffice.
\end{remark}

\begin{example}\label{example:period_matrix}
(1) Let $\Gamma{}$ be the Schottky group in Example \ref{example:good_fundamental_domain}(1). Choose the same good fundamental domain, with  $B_1=B(4,1/9)$, $B_1'=B(1,1/9)$, $B_2=B(5,1/9)$, and $B_2'=B(2,1/9)$.  The four balls correspond to four points in the tree $(\pro^1)^{an}$. We need to find the pairwise distances between the points $P_1$, $P'_1$, $P_2$, and $P'_2$ in $(\pro^1)^{an}$. Since the smallest ball containing both $B^+_1$ and $B'^+_1$ is $B^+(1,1/3)$, both $P_1$ and $P_1'$ are distance $\val((1/3)/(1/9))=\val(3)=1$ from the point corresponding to $B^+(1,1/3)$, so $P_1$ and $P_1'$ are distance $2$ from one another.  Similar calculations give distances of $2$ between $P_2$ and $P_2'$, and of $4$ between $P_1$ or $P_1'$ and $P_2$ or $P_2'$.  In fact, the distance between $P_i$ and $P_i'$ equals the difference in the valuations of the two eigenvalues of $\gamma{}_i$. This allows us to construct the subtree of $(\mathbb{P}^1)^{an}$ spanned by $P_1,P_2,P_1',P_2'$ as illustrated in Figure \ref{figure:genus2_example1}. The minimum distance between them is $c=2$. To approximate $Q_{11}$, we take $a=10$ and $z=19$. To compute $Q$ up to relative precision $O(p^{10})$, we need $2m\geq 10$ (this is the equation $cm\geq n$ from Remark \ref{remark:choosing_m}), so choosing $m=5$ works. The output of the algorithm is $Q_{11}=(\dotsc{}220200000100)_3$. Similarly, we can get the other entries in the matrix $Q$:

\begin{equation*}
Q =
\begin{bmatrix}
(\dotsc{}220200000100)_3 & (\dotsc{}0101010101)_3 \\
(\dotsc{}0101010101)_3 & (\dotsc{}220200000100)_3
\end{bmatrix}.
\end{equation*}

(2) Let $\Gamma{}$ be the Schottky group in Example \ref{example:good_fundamental_domain}(2). Choose the same good fundamental domain. Again, we need $m=5$ for relative precision $O(p^{10})$. The algorithm outputs
\begin{equation*}
Q =
\begin{bmatrix}
(\dotsc{}12010021010000)_3 & (\dotsc{}002000212200)_3 \\
(\dotsc{}002000212200)_3 & (\dotsc{}12010021010000)_3
\end{bmatrix}.
\end{equation*}

(3) Let $\Gamma{}$ be the Schottky group in Example \ref{example:good_fundamental_domain}(3). Choose the same good fundamental domain. The minimum distance between the corresponding points in $(\pro^1)^{an}$ is $2$, so we may take $m=10/2=5$ to have relative precision up to $O(p^{10})$.  Our algorithm outputs
\begin{equation*}
Q =
\begin{bmatrix}
(\dotsc{1 1 2 0 1 0 0 0 0 1 0 0 0 0})_3 & (\dotsc{1 2 0 2 0 0 2 2 2 1 0})_3 &(\dotsc{2 0 0 2 0 0 0 2 1 2 0})_3 \\
(\dotsc{1 2 0 2 0 0 2 2 2 1 0})_3 & (\dotsc{1 0 1 0 1 0 1 0 0 1 0 0 0 0})_3 &(\dotsc{0 2 0 2 0 1 1 2 0 . 1})_3  \\
(\dotsc{2 0 0 2 0 0 0 2 1 2 0})_3 & (\dotsc{0 2 0 2 0 1 1 2 0 . 1})_3 &(\dotsc{2 1 0 1 0 1 0 0 0 1 0 0 0 0})_3 
\end{bmatrix}.
\end{equation*}

(4) Let $K=\mathbb{C}_3$ and $\Gamma{}$ be the group generated by
\begin{equation*}
\gamma{}_1 = \begin{bmatrix} -5 & 32 \\ -8 & 35 \end{bmatrix}, \gamma{}_2 = \gamma{}_1^{100}\begin{bmatrix} -13 & 80 \\ -8 & 43 \end{bmatrix}
\end{equation*}
The group is the same as in part (1) of this set of examples, but the generators are not in good position. To achieve the same precision, $m$ needs to be up to $100$ times greater than in part (1), because the $\gamma{}_2$ in part (1) now has a reduced word of length $101$. Since the running time grows exponentially in $m$, it is not feasible to approximate $Q$ using Algorithm \ref{algperiodmatrix} with these generators as input.
\end{example}

\subsection{The Abstract Tropical Curve}

This subsection deals with the problem of constructing the corresponding abstract tropical curve of a Schottky group over $K$, together with some data on its homotopy group. This is a relatively easy task, assuming that the given generators $\gamma{}_1,\dotsc{},\gamma{}_g$ are in good position, and that we are also given a fundamental domain $F=\mathbb{P}^1\backslash{}(B_1\cap{}\dotsb{}\cap{}B_g\cap{}B_1'\cap{}\dotsb{}\cap{}B_g')$. Without loss of generality, we may assume that $\infty{}\in{}F^{\circ{}}$. Let $P_1,\dotsc{},P_g,P_1',\dotsc{},P_g'\in (\pro^1)^{an}$ be the Gaussian points of the disks $B_1,\dotsc{},B_g,B_1',\dotsc{},B_g'$.

Let $R_g$ be the \emph{rose graph} on $g$ leaves (with one vertex and $g$ loops), and let $r_1,\ldots,r_g$ be the loops.  A homotopy equivalence $h:R^g\rightarrow G$ must map $r_1,\ldots,r_g$ to $g$ loops of $G$ that generate $\pi_1(G)$, so to specify $h$ it will suffice to label $g$ such loops of $G$ with $\{s_1,\ldots, s_g\}$ and orientations.  It is for this reason that we call $h$ a \emph{marking of $G$}.

\begin{algorithm}[Abstract Tropical Curve Construction]\label{algorithm:abstract_tropical_curve}\quad{}
{\rm
\begin{algorithmic}[1]

\Require Matrices $\gamma_1,\ldots \gamma_g\in\Q_p^{2\times 2}$ generating a Schottky group $\Gamma$, together with a good fundamental domain $F=\mathbb{P}^1\backslash{}(B_1\cap{}\dotsb{}\cap{}B_g\cap{}B_1'\cap{}\dotsb{}\cap{}B_g'$).
\Ensure The triple $(G,\ell,h)$ with  $(G,\ell)$ the abstract tropical curve as a metric graph with a $h$ a marking  presented as $g$ labelled oriented loops of $G$.

\State Construct the subtree in $(\mathbb{P}^1)^{an}$ spanned by $P_1,\dotsc{},P_g,P_1',\dotsc{},P_g'$, including lengths.
\State Label the unique shortest path from $P_i$ to $P_i'$ as $s_i$, remembering orientation.
\State Identify each $P_i$ with $P_i'$, and declare the length of the new edge containing $P_i=P_i'$ to be the sum of the lengths of the edges that were joined to form it.
\State Define $h$ by the labels $s_i$, with each $s_i$ now an oriented loop.
\State \Return the resulting labeled metric graph $(G,\ell, h)$.

\end{algorithmic}
}
\end{algorithm}

\begin{proof}
The proof is essentially given in \cite[I 4.3]{GP}.
\end{proof}

\begin{remark}  It's worth noting that this algorithm can be done by hand if a good fundamental domain is known.  If $P_1,P_2\in(\pro^1)^{an}$ are the points corresponding to the disjoint closed balls $B(a_1,r_1)^+$ and $B(a_2,r_2)^+$, then the distance between $P_1$ and $P_2$ is just the sum of their distances from $P_3$ corresponding to $B(a_3,r_3)^+$, where $B(a_3,r_3)^+$ is the smallest closed  ball containing both $a_1$ and $a_2$.  The distance between $P_i$ and $P_3$ is just $\val(r_3/r_i)$ for $i=1,2$.  Once all pairwise distances are known, constructing $(G,\ell)$ is simple.  Finding $h$ is simply a matter of drawing the orientation on the loops formed by each pair $(P_i,P_i')$ and labeling that loop $s_i$.  This process is illustrated three times in Example \ref{example:abstract_tropical_curve}.
\end{remark}

\begin{remark}\label{remark:teichmuller} The space parameterizing labelled metric graphs $(G,\ell,h)$  (identifying those with markings that are homotopy equivalent) is called \emph{Outer space}, and is denoted $X_g$.  It is shown in \cite{CMV} that $X_g$ sits inside tropical Teichm\"{u}ller space as a dense open set, so Algorithm \ref{algorithm:abstract_tropical_curve} can be viewed as computing a point in  tropical Teichm\"{u}ller space.
\end{remark}

\begin{example}\label{example:abstract_tropical_curve} 
(1) Let $\Gamma{}$ be the Schottky group in Example \ref{example:good_fundamental_domain}(1). Choose the same good fundamental domain, with  $B_1=B(4,1/9)$, $B_1'=B(1,1/9)$, $B_2=B(5,1/9)$, and $B_2'=B(2,1/9)$. We have constructed the subtree of $(\mathbb{P}^1)^{an}$ spanned by $P_1,P_2,P_1',P_2'$ as illustrated in Figure \ref{figure:genus2_example1} in Example \ref{example:period_matrix}(1). After identifying $P_1$ with $P_1'$ and $P_2$ with $P_2'$, we get the ``dumbbell" graph shown in Figure \ref{figure:genus2_example1}, with both loops having length $2$ and the connecting edge having length $2$.

(2) Let $\Gamma{}$ be the Schottky group in Example \ref{example:good_fundamental_domain}(2). Choose the same good fundamental domain. The subtree of $(\mathbb{P}^1)^{an}$ spanned by $P_1,P_2,P_1',P_2'$ is illustrated in Figure \ref{figure:genus2_example2}. After identifying $P_1$ with $P_1'$ and $P_2$ with $P_2'$, we get the ``theta" graph shown in Figure \ref{figure:genus2_example2}, with two edges of length $2$ and one edge of length $2$.

(3) Let $\Gamma$ be the Schottky group in Example  \ref{example:good_fundamental_domain}(3).
The subtree of $(\mathbb{P}^1)^{an}$ spanned by $P_1,P_2,P_3,P_1',P_2',P_3'$ is illustrated in Figure \ref{figure:genus3_example1}. After identifying $P_1$ with $P_1'$, $P_2$ with $P_2'$ and $P_3$ with $P_3'$, we get the ``honeycomb" graph shown in Figure \ref{figure:genus3_example1}, with interior edges of length $1$ and exterior edges of length $2$.
\end{example}

\begin{figure}
\caption{The tree in Example \ref{example:abstract_tropical_curve}(1), and the abstract tropical curve.}
\centering
\includegraphics[scale=1]{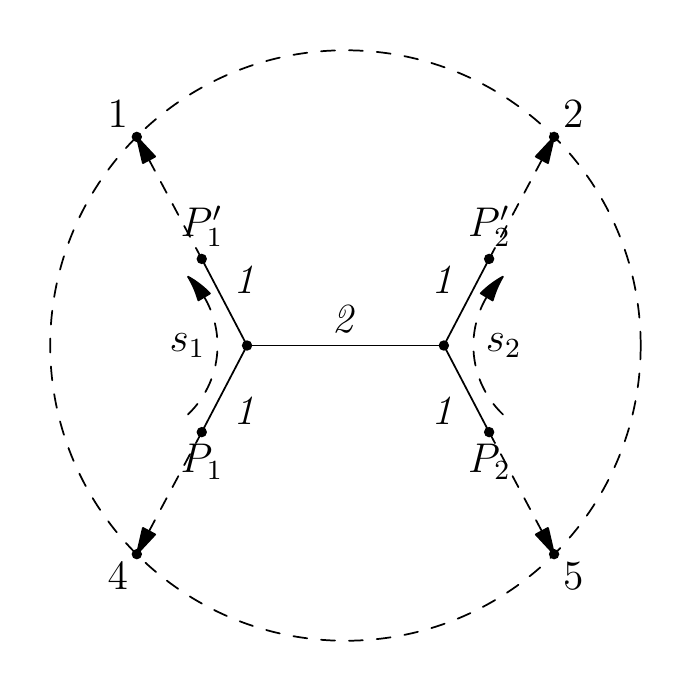}
\includegraphics[scale=1]{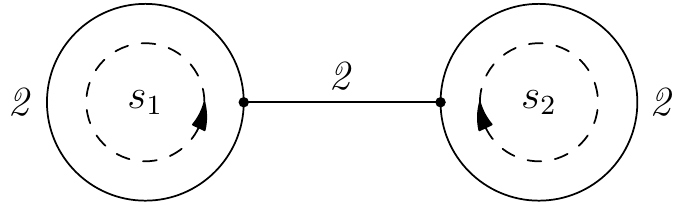}
\label{figure:genus2_example1}
\end{figure}

\begin{figure}
\caption{The tree in Example \ref{example:abstract_tropical_curve}(2), and the abstract tropical curve.}
\centering
\includegraphics[scale=1]{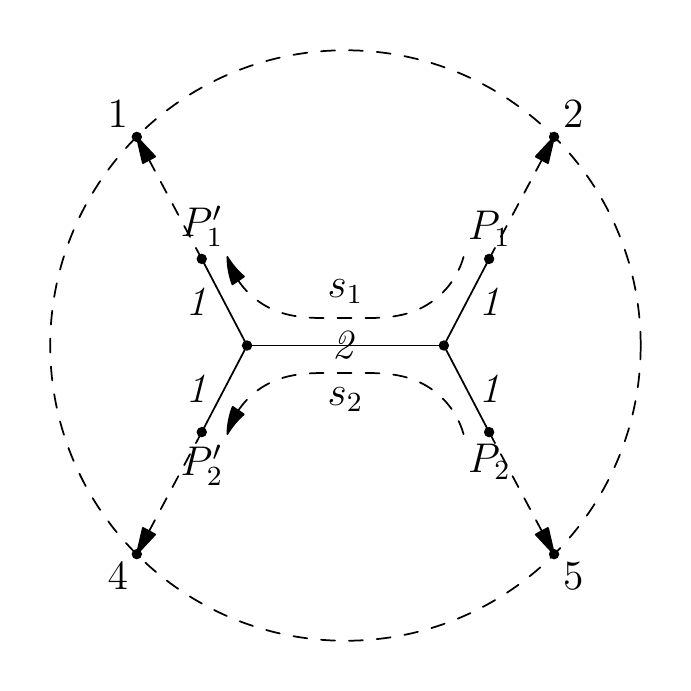}
\includegraphics[scale=1]{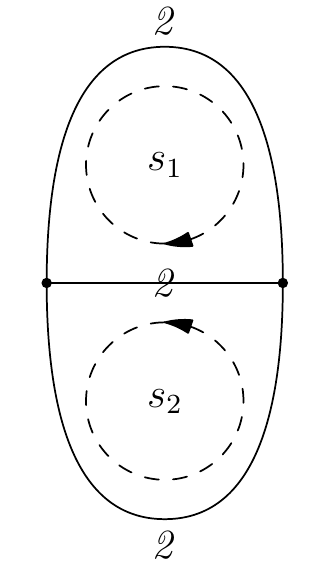}
\label{figure:genus2_example2}
\end{figure}

\begin{figure}
\caption{The tree in Example \ref{example:abstract_tropical_curve}(3), and the abstract tropical curve.}
\centering
\includegraphics[scale=0.8]{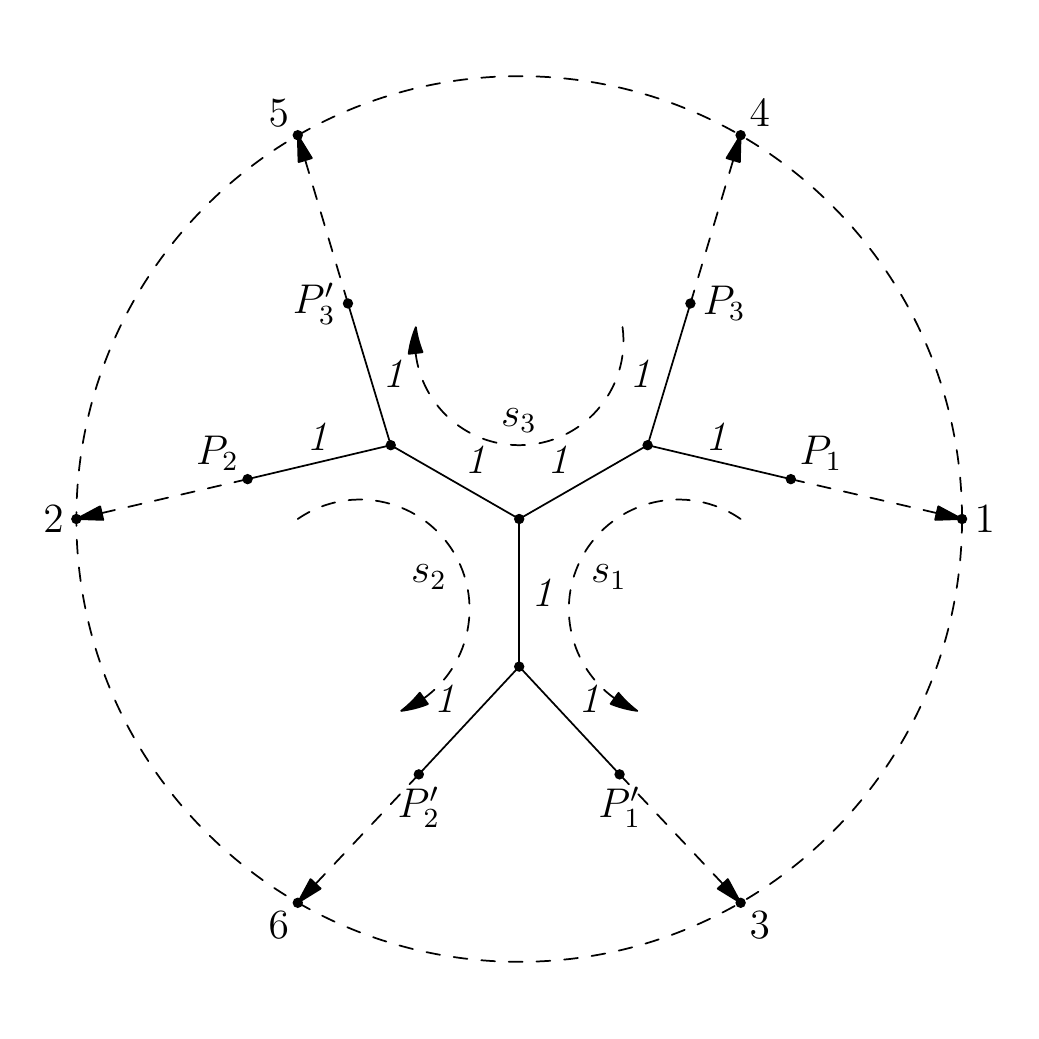}
\includegraphics[scale=1]{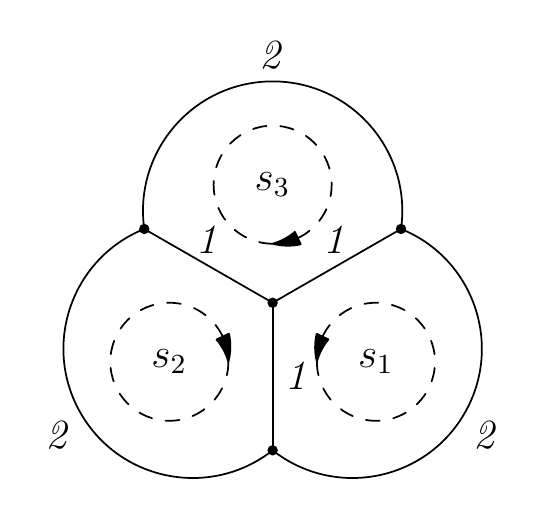}
\label{figure:genus3_example1}
\end{figure}

\subsection{Canonical Embeddings}

From \cite[VI.4]{GP}, we have that
\begin{equation*}
\omega{}_i(z) := w_i(z)dz = \frac{u_{\gamma{}_i}'(z)}{u_{\gamma{}_i}(z)}dz
\end{equation*}
are $g$ linearly independent analytic differentials on $\Omega$ that are invariant under the action of $\Gamma{}$. Therefore, they define $g$ linearly independent differentials on $C=\Omega/\Gamma{}$. Gerritzen and van der Put \cite[VI.4]{GP} also state that these form a basis of the space of $\Gamma{}$-invariant analytic differentials. Since the space of algebraic differentials on $C$ has dimension $g$, it must be generated by these $g$ differentials. Therefore, the canonical embedding has the following form:
\begin{align*}
C&\to{}\mathbb{P}^{g-1},\\
z&\mapsto{}\left(\frac{u_{\gamma{}_1}'(z)}{u_{\gamma{}_1}(z)}:\dotsc{}:\frac{u_{\gamma{}_g}'(z)}{u_{\gamma{}_g}(z)}\right).
\end{align*}

It therefore suffices to approximate the derivative $u_{\alpha{}}'(z)$. A na\"{i}ve approach is to consider the approximation
\begin{equation*}
u_{\alpha{}}'(z)\approx{}\frac{u_{\alpha{}}(z+h)-u_{\alpha{}}(z)}{h}.
\end{equation*}
We can do better by taking advantage of the product form of $u_{\alpha{}}(z)$:
\begin{align*}
u_{\alpha{}}'(z) &= \frac{d}{dz}\prod_{\gamma{}\in{}\Gamma{}}\frac{z-\gamma{}a}{z-\gamma{}\alpha{}a} \\
&= \sum_{\gamma{}\in{}\Gamma{}}\left(\frac{d}{dz}\left(\frac{z-\gamma{}a}{z-\gamma{}\alpha{}a}\right)\prod_{\gamma{}'\in{}\Gamma{},\gamma{}'\neq{}\gamma{}}\frac{z-\gamma{}'a}{z-\gamma{}'\alpha{}a}\right) \\
&= u_{\alpha{}}(z) \sum_{\gamma{}\in{}\Gamma{}}\frac{d}{dz}\left(\frac{z-\gamma{}a}{z-\gamma{}\alpha{}a}\right)\left(\frac{z-\gamma{}a}{z-\gamma{}\alpha{}a}\right)^{-1} \\
&= u_{\alpha{}}(z) \sum_{\gamma{}\in{}\Gamma{}}\frac{\gamma{}a-\gamma{}\alpha{}a}{(z-\gamma{}a)(z-\gamma{}\alpha{}a)}.
\end{align*}

\begin{algorithm}[Canonical Embedding]\label{algcanonical}\quad{}
{\rm
\begin{algorithmic}[1]

\Require Matrices $\gamma_1,\ldots ,\gamma_g\in\Q_p^{2\times 2}$ generating a Schottky group $\Gamma$ in good position, an element $z\in{}K$, and an integer $n$ to determine precision.
\Ensure An approximation for the image of $z$ under \emph{the} canonical embedding $\Omega{}/\Gamma{}\to{}\mathbb{P}^{g-1}$ determined by the choice of generators.

\State Based on $n$, choose a suitable positive integer $m$ as described in Remark \ref{remark:choosing_m_canonical}.
\For {$i=1$ to $g$}
\State Choose a suitable element $a\in{}K$ as described in Proposition \ref{proposition:canonical_approximation}.
\State Compute
\begin{equation*}
w_i = \sum_{\gamma{}\in{}\Gamma{}_m}\frac{\gamma{}a-\gamma{}\gamma{}_ia}{(z-\gamma{}a)(z-\gamma{}\gamma{}_ia)}.
\end{equation*}
\EndFor
\State \Return $(w_1:\dotsb{}:w_g)$.
\end{algorithmic}
}
\end{algorithm}

With appropriate choice of $a$, we can provide a lower bound on the precision of the result in terms of $m$. Fortunately, we can choose different values of $a$ to approximate
\begin{equation*}
\sum_{\gamma{}\in{}\Gamma{}}\frac{\gamma{}a-\gamma{}\gamma{}_ia}{(z-\gamma{}a)(z-\gamma{}\gamma{}_ia)}
\end{equation*}
for different $\gamma{}_i$. As in Proposition \ref{statement:approximation_fraction}, we choose $a\in{}B_i'^+\backslash{}B_i'$ to ensure that both $a$ and $\gamma{}_ia$ are in $F$.

\begin{proposition}\label{proposition:canonical_approximation}
If we choose $a\in{}B_i'^+\backslash{}B_i'$ in Algorithm \ref{algcanonical}, and assuming $z\in{}F$, then
\begin{equation*}
\left|\sum_{\gamma{}\in{}\Gamma{}_m}\frac{\gamma{}a-\gamma{}\gamma{}_ia}{(z-\gamma{}a)(z-\gamma{}\gamma{}_ia)}-\frac{u_{\gamma{}_i}'(z)}{u_{\gamma{}_i}(z)}\right|\leq{}p^{-mc-\log_p(d)},
\end{equation*}
where $c$ is the minimum pairwise distance between $P_1,\dotsc{},P_g,P_1',\dotsc{},P_g'$, and $d$ is the minimum diameter of $B_1,\dotsc{},B_g,B_1',\dotsc{},B_g'$.
\end{proposition}

\begin{proof}
Let $\gamma{}\in{}\Gamma{}$ have reduced word $\gamma{}=h_1h_2\dotsb{}h_k$. We have seen in the proof of Proposition \ref{statement:approximation_fraction} that $|\gamma{}a-\gamma{}\gamma{}_ia|\leq{}p^{-(k-1)c}\mathrm{diam}(B^+)$, $|z-\gamma{}a|\geq{}\mathrm{diam}(B^+)$ and $|z-\gamma{}\gamma{}_ia|\geq{}\mathrm{diam}(B^+)$, where $B$ is one of $B_1,\dotsc{},B_g,B_1',\dotsc{},B_g'$. Thus,
\begin{align*}
\left|\frac{\gamma{}a-\gamma{}\gamma{}_ia}{(z-\gamma{}a)(z-\gamma{}\gamma{}_ia)}\right| &\leq{} \frac{p^{-(k-1)c}\mathrm{diam}(B^+)}{\mathrm{diam}(B^+)^2} \\
&\leq{} p^{-(k-1)c}d^{-1}
\\&=p^{-(k-1)c-\log_p(d)}.
\end{align*}
Since the difference between our approximation and the true value is the sum over terms where $\gamma{}$ has reduced words of length $\geq{}m+1$, we conclude that the error has absolute value at most $p^{-mc-\log_p(d)}$.
\end{proof}

In the last proposition, we assumed $z\in{}F$. If $z\notin{}F$, we can do an extra step and replace $z$ by some $\gamma{}z$ such that $\gamma{}z\in{}F$, with the help of Subroutine \ref{algorithm:reduction_fundamental_domain}. This step does not change the end result because the theta functions are invariant under the action of $\Gamma{}$.

\begin{remark}\label{remark:choosing_m_canonical}  If we wish to use Algorithm \ref{algcanonical} to compute a period matrix $Q$ with accuracy up to the $n^{th}$ $p$-adic digit, we must first compute $c$ and $d$. Recall that $c$ is defined to be the minimum distance between pairs of the points $P_1,\ldots ,P_g,P_1',\ldots,P_g'\in (\pro^{1})^{an}$ corresponding to the balls $B^+_1,\ldots, B^+_g,B'^+_1\ldots,B'^+_g$ that characterize our good fundamental domain, and $d$ is the minimum diameter of $B_1,\ldots, B_g,B_1',\ldots,B_g'$.  Once we have computed $c$ and $d$, then by Proposition \ref{proposition:canonical_approximation}  we must choose $m$ such that $p^{-mc}d^{-1}\leq p^{-n}$.  We could also think of it as choosing $m$ such that $mc+\log_p(d)\geq n$.
\end{remark}

\begin{remark}  As was the case with Algorithm \ref{algperiodmatrix}, we may run Algorithm \ref{algcanonical} even if the input generators are not in good position, and it will approximate images of points in the canonical embedding.  However, we will not have control over the rate of convergence, which will in general be very slow.
\end{remark}

\begin{example}\label{example:canonical_embedding}
Let $\Gamma{}$ be the Schottky group in Example \ref{example:good_fundamental_domain}(3). Choose the same good fundamental domain. We will compute the image of the field element $17$ under the canonical embedding (we have chosen $17$ as it is in $\Omega$ for this particular $\Gamma$). The minimum diameter is $d=1/9$, and the minimum distance is $c=2$. To get absolute precision to the order of $p^{-10}$, we need $p^{-mc}d^{-1}\leq{}p^{-10}$, i.e. $m\geq{}6$. Applying Algorithm \ref{algcanonical} with $m=6$ gives us the following point in $\mathbb{P}^2$:
\begin{equation*}
((\dotsc{}2100012121)_3:(\dotsc{}2211022001.1)_3:(\dotsc{}2221222111.1)_3).
\end{equation*}

This point lies on the canonical embedding of the genus $3$ Mumford curve $\Omega/\Gamma$.  Any genus $3$ curve is either a hyperelliptic curve or a smooth plane quartic curve.  However, it is impossible for a hyperelliptic curve to have the skeleton in Figure \ref{figure:genus3_example1} (see \cite[Theorem 4.15]{Ch}), so $\Omega/\Gamma$ must be a smooth plane quartic curve. Its equation has the form
\begin{align*}
&C_1x^4+C_2x^3y+C_3x^3z+C_4x^2y^2+C_5x^2yz+C_6x^2z^2 \\
&\quad{}C_7xy^3+C_8xy^2z+C_9xyz^2+C_{10}xz^3+C_{11}y^4+C_{12}y^3z+C_{13}y^2z^2+C_{14}yz^3+C_{15}z^4.
\end{align*}
Using linear algebra over $\mathbb{Q}_3$, we can solve for its $15$ coefficients by computing $14$ points on the curve and plugging them into the equation. The result is
\begin{equation*}
\begin{array}{lll}
C_1=1, & C_2=(\dotsc{}11101)_3, & C_3=(\dotsc{}00211)_3, \\
C_4=(\dotsc{}1020.2)_3, & C_5=(\dotsc{}110.21)_3, & C_6=(\dotsc{}1002.1)_3, \\
C_7=(\dotsc{}122)_3, & C_8=(\dotsc{}222.02)_3, & C_9=(\dotsc{}222.02)_3, \\
C_{10}=(\dotsc{}21101)_3, & C_{11}=(\dotsc{}2122)_3, & C_{12}=(\dotsc{}2201)_3, \\
C_{13}=(\dotsc{}0202.2)_3, & C_{14}=(\dotsc{}10102)_3, & C_{15}=(\dotsc{}01221)_3.
\end{array}
\end{equation*}
For the Newton subdivision and tropicalization of this plane quartic, see the following subsection, in which we consider the interactions of the three algorithms of Section \ref{section:main_algorithms}.

\end{example}

\subsection{Reality Check:  Interactions Between The Algorithms}  We close Section \ref{section:main_algorithms} by checking that the three algorithms give results consistent with one another and with some mathematical theory.  We will use our running example of a genus 3 Mumford curve from Examples \ref{example:period_matrix}(3), \ref{example:abstract_tropical_curve}(3), and \ref{example:canonical_embedding}, for which we have computed a period matrix of the Jacobian, the abstract tropical curve, and a canonical embedding. 

First we will look at the period matrix and the abstract tropical curve, and verify that these outputs are consistent.   
Recall that for the period matrix $Q$ of $\text{Jac}(\Omega/\Gamma)$, we have
$$Q_{ij}=\frac{u_{\gamma_i}(z)}{u_{\gamma_i}(\gamma_j z)}.$$
Motivated by this, we define
\begin{align*}
Q: \Gamma\times\Gamma\rightarrow K^*\\
(\alpha,\beta)\mapsto \frac{u_\alpha(z)}{u_\alpha(\beta z)},
\end{align*} 
where our choice of $z\in\Omega$ does not affect the value of $Q(\alpha,\beta)$.  (Note that $Q_{ij}=Q(\gamma_i,\gamma_j)$.)  As shown in \cite[VI, 2]{GP}, the kernel of $Q$ is the commutator subgroup $[\Gamma,\Gamma]$ of $\Gamma$, and $Q$ is symmetric and positive definite (meaning $|Q(\alpha,\alpha)|<1$ for any $\alpha\not\equiv\left[\begin{matrix}1&0\\ 0&1\end{matrix}\right]\mod[\Gamma,\Gamma]$).  Moreover,  the following theorem holds (see \cite[Theorem 6.4]{Pu2}).

\begin{theorem} \label{metricgraphimportance}
Let $G$ be the abstract tropical curve of $\Omega/\Gamma$, and let $\pi_1(G)$ be its homotopy group, treating $G$ as a topological space.  There is a canonical isomorphism $\phi{}\colon{}\Gamma{}^{ab}\to{}\pi{}_1(G)^{ab}$ such that $\text{val}(Q(\gamma{},\gamma{}'))=\langle{}\phi{}(\gamma{}),\phi{}(\gamma{}')\rangle{}$, where  $\langle p_1 ,p_2 \rangle{}$ denotes the shared edge length of the oriented paths $p_1$ and $p_2$.
\end{theorem}

The map $\phi$ is made very intuitive by considering the construction of $G$ in Algoirthm \ref{algorithm:abstract_tropical_curve}:  a generator $\gamma_i$ of $\Gamma$ yields two points $P_i,P_i'\in (\pro^1)^{an}$ (corresponding to balls containing the eigenvalues of $\gamma_i$), and these points are glued together in constructing $G$.  So $\gamma_i$ corresponds to a loop around the cycle resulting from this gluing; after abelianization, this intuition is made rigorous.

Consider the matrix $Q$ computed in Example \ref{example:period_matrix}(3).  Worrying only about valuations, we have
$$\val(Q)=\left[\begin{matrix}
4 & 1 & 1\\
1& 4 & -1\\
1 & -1 & 4
\end{matrix}\right].$$
For $i=1,2,3$, let $s_i$ be the oriented loop in $G$ arising from gluing $P_i$ and $P_i'$.   In light of Theorem \ref{metricgraphimportance}, we expect to find shared edge lengths
$$\left<s_1,s_1\right>=\left<s_2,s_2\right>=\left<s_3,s_3\right>=4,$$
$$\left<s_1,s_2\right>=\left<s_1,s_3\right>=1,$$
and\[\left<s_2,s_3\right>=-1.\]
That is, each cycle length should be $4$, and the common edge of each distinct pair of cycles should have length $1$, with the orientation of $s_1$ agreeing with the orientation of $s_2$ (respectively, $s_3$) on the shared edge and the orientation of $s_2$ disagreeing with the orientation of $s_3$ on the shared edge.  This is indeed what we found in Example \ref{example:abstract_tropical_curve}(3), with edge lengths and orientations shown in Figure \ref{figure:genus3_example1}.  This example has shown how the outputs of Algorithms \ref{algperiodmatrix} and \ref{algorithm:abstract_tropical_curve} can be checked against one another.

We will now consider the relationship between the abstract tropical curve and the canonical embedding for this example.  In particular, we will compute a tropicalization of the curve from the canonical embedding and see how this relates to the abstract tropical curve.

To compute the tropicalization of the curve, we will start with the quartic planar equation computed in Example \ref{example:canonical_embedding}.  The {\it Newton polytope} of this quartic is a triangle with side length $4$. We label each integral point inside or on the boundary of the Newton polytope by the valuation of the coefficient of the corresponding term, ignoring the variable $z$. For example, the point $(1,2)$ is labeled $-2$ because the valuation of the coefficient of $xy^2z$ is $\mathrm{val}(C_8)=-2$. We then take the {\it lower convex hull}, giving a subdivision of the Newton polytope as shown in Figure \ref{figure:genus3_example1_newton_polytope}. 
The tropicalization of the curve is combinatorially the dual graph of this polytope, and using the max convention of tropical geometry it sits in $\R^2$ as shown in Figure \ref{figure:genus3_example1_newton_polytope}, with the common point of the three cycles at $(0,0)$. 

Let us compare the cycles in the tropicalization with the cycles in the abstract tropical curve.  We know from \cite[\textsection 6.23]{BPR} that this tropicalization is faithful since all vertices are trivalent and are adjacent to at least one edge of weight one.
This means that lattice lengths on the tropicalization should agree with lengths of the abstract tropical curve.  Each cycle in the tropicalization has five edges, and for each cycle two edges are length $\frac{1}{2}$ and three are length $1$.  This gives a length of $4$, as we'd expect based on Example \ref{example:abstract_tropical_curve}(3).  Moreover, each shared edge has lattice length $1$, as was the case in the abstract tropical curve.  Thus we have checked the outputs of Algorithms \ref{algorithm:abstract_tropical_curve} and \ref{algcanonical} against one another.

\begin{figure}
\caption{The Newton polytope of the plane quartic curve in Example \ref{example:canonical_embedding}, and the corresponding tropical curve in $\R^2$ (drawn using the max convention).  Each edge of infinite length has weight $2$, and all other edges have weight $1$.}
\centering
\includegraphics[scale=0.7]{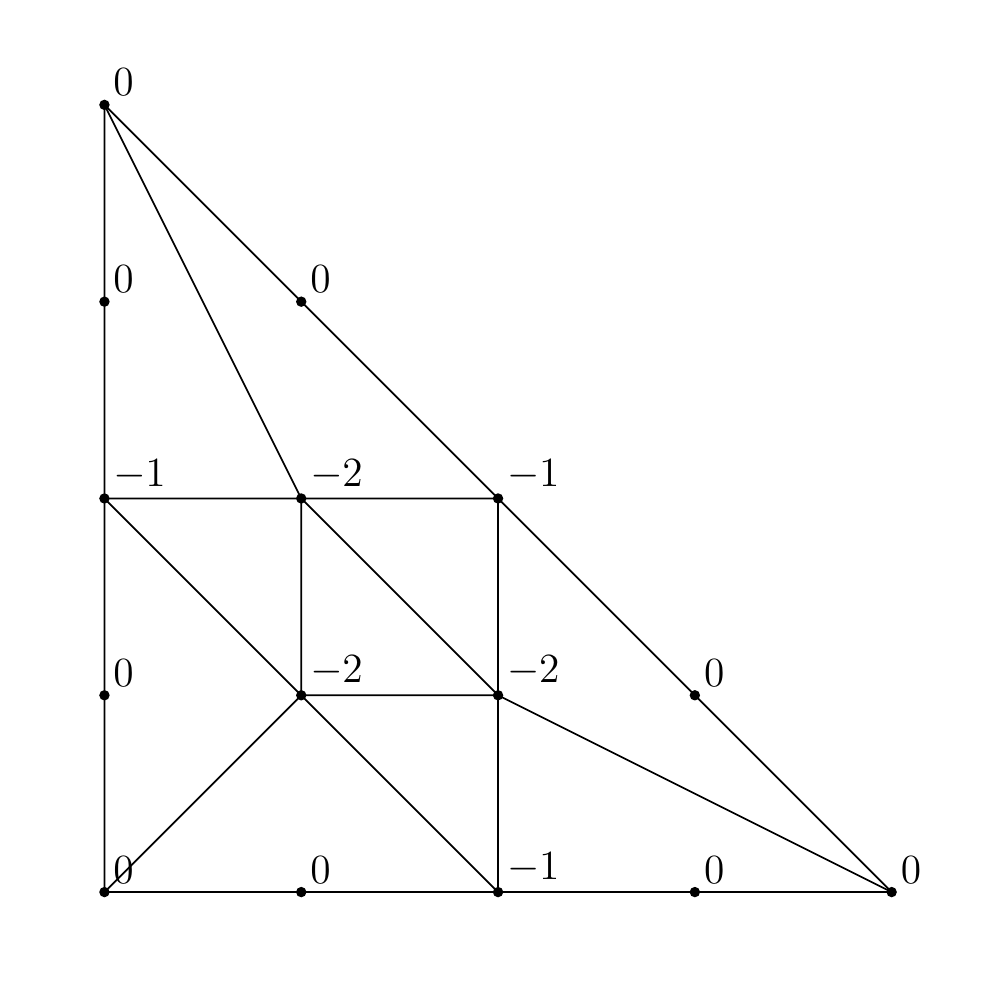}
\includegraphics[scale=0.7]{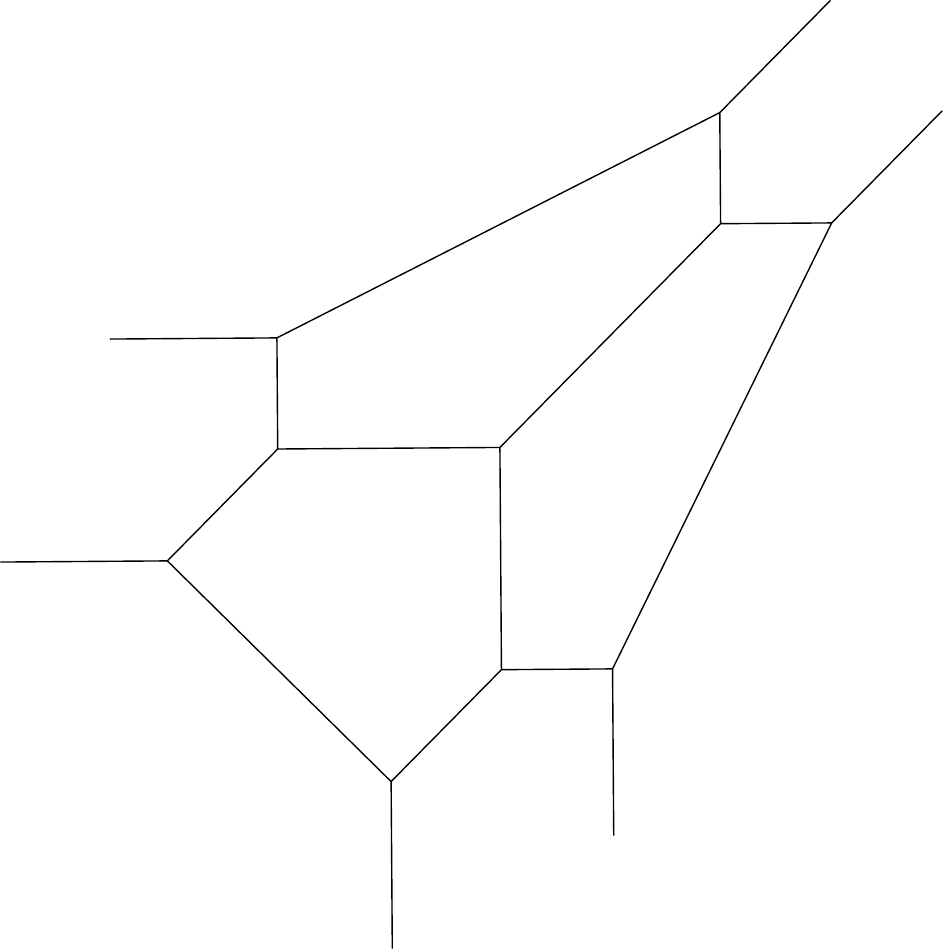}
\label{figure:genus3_example1_newton_polytope}
\end{figure}

\section{From Generators in Bad Position to Generators in Good Position}\label{goodpositionalgorithmsection}

The previous section describes several algorithms that compute various objects from a set of free generators of a Schottky group, assuming that the generators are in good position, and (in Algorithm \ref{algorithm:abstract_tropical_curve}) that a good fundamental domain is given together with the generators. This content of this section is what allows us to make this assumption. We give an algorithm (Algorithm \ref{algorithm:main_4}) that takes an arbitrary set of free generators of a Schottky group and outputs a set of free generators that are in good position, together with a good fundamental domain.  This algorithm can be modified as described in Remark \ref{remark:schottky_test} to perform a ``Schottky test''; in particular, given a set of $g$ invertible matrices generating a group $\Gamma$, the modified algorithm will either
\bi

\item return a set of $g$ free generators of $\Gamma$ in good position together with a good fundamental domain, which is a certificate that $\Gamma$ is Schottky;

\item return a relation satisfied by the input matrices, which is a certificate that the generators do not freely generate the group; or

\item return a non-hyperbolic, non-identity matrix $\gamma\in\Gamma$, which is a certificate that $\Gamma$ is not Schottky.

\ei
Before presenting Algorithm \ref{algorithm:main_4}, we will first develop some theory for trees, and then define useful subroutines.  Our starting point is a remark in Gerritzen and van der Put's book:

\begin{proposition}\cite[III 2.12.3]{GP}
Let $\Gamma{}$ be Schottky and $\Sigma{}$ and $\Omega{}$ be as usual. Let $T(\Sigma{})$ be the subtree of $(\mathbb{P}^1)^{an}$ spanned by $\Sigma{}$. Then the minimal skeleton of $\Omega{}/\Gamma{}$ is isomorphic to~$T(\Sigma{})/\Gamma{}$.
\end{proposition}

This statement is essential for our algorithm, because it helps reducing problems involving $(\mathbb{P}^1)^{an}$ to problems involving the much simpler tree $T(\Sigma{})$. Though $T(\Sigma{})$ is not finite, it is a finitely branching tree:  it consists of vertices and edges such that each vertex is connected with finitely many edges. A good fundamental domain in $(\mathbb{P}^1)^{an}$ can be obtained from a {\it good fundamental domain} in $T(\Sigma{})$, defined as follows:

\begin{defn}
A \emph{principal subtree} $T$ of $T(\Sigma{})$ is a connected component of $T(\Sigma{})\backslash{}\{e\}$ for some edge $e$ of $T(\Sigma{})$. An \emph{extended principal subtree} is $T^+=T\cup{}\{e\}$.
\end{defn}

\begin{defn}
A \emph{good fundamental domain} $S$ in $T(\Sigma{})$ for a set of free generators $\gamma{}_1,\dotsc{},\gamma{}_g$ of $\Gamma{}$ is the complement of $2g$ principal subtrees $T_1,\dotsc{},T_g,T_1',\dotsc{},T_g'$, such that $T_1^+,\dotsc{},T_g^+,$ $T_1'^+,\dotsc{},T_g'^+$ are disjoint, and that $\gamma{}_i(T(\Sigma{})\backslash{}T_i'^+)=T_i$ and $\gamma{}_i^{-1}(T(\Sigma{})\backslash{}T_i^+)=T_i'$. The \emph{interior} of $S$ is $S^{\circ{}}=T(\Sigma{})\backslash{}(T_1^+\cup{}\dotsb{}\cup{}T_g^+\cup{}T_1'^+\cup{}\dotsb{}\cup{}T_g'^+)$. The \emph{boundary} of $S$ is~$S\backslash{}S^{\circ{}}$.
\end{defn}

In other words, $S$ is a connected finite subtree of $T(\Sigma{})$ with $2g$ boundary edges $(R_i,Q_i)$ and $(R_i',Q_i')$, where $Q_i,Q_i'\notin{}S$, such that $\gamma{}_i(R_i',Q_i')=(Q_i,R_i)$. Given this data, the principal subtree $T_i$ (resp. $T_i'$) is the connected component of $T(\Sigma{})\backslash{}(R_i,Q_i)$ (resp. $T(\Sigma{}\backslash{}(R_i',Q_i')$) that is disjoint from $S$. Given a good fundamental domain $S$ in $T(\Sigma{})$, one can find a good fundamental domain in $(\mathbb{P}^1)^{an}$ as follows. Without loss of generality, we may assume that the retraction of $\infty{}$ to $T(\Sigma{})$ is in the interior of $S$. Then, $Q_i$ and $R_i$ correspond to two nested balls $B(a_i,r_i)^+\subset{}B(a_i,R_i)^+$. Define $B_i=B(a_i,\sqrt{r_iR_i})$. Define $B_i'$ similarly.

\begin{proposition}\label{statement:graph_to_analytic_fundamental_domain}
Let $B_i,B_i'$ be as above. Then $F=(\mathbb{P}^1)^{an}\backslash{}(B_1\cup{}\dotsb{}\cup{}B_g\cup{}B_1'\cup{}\dotsb{}\cup{}B_g')$ is a good fundamental domain.
\end{proposition}

\begin{proof}
Let $Q_i,R_i$ be as above. Let $P_i$ be the midpoint of the segment of the boundary edge $(Q_i,R_i)$. Then, $P_i$ corresponds to the ball $B_i^+$. Let $\pi{}$ denote the retraction from $(\mathbb{P}^1)^{an}$ to $T(\Sigma{})$. Again, we may assume $\pi{}(\infty{})$ is in the interior of $S$. For any $P\in{}B_i^+$, the unique path from $P$ to $\infty{}$ passes through $P_i$. Therefore, $\pi{}(P)$ lies on the union of $T_i$ with the segment $(P_i,Q_i)$, which is a subset of $T_i^+$. Hence, the condition that $T_i^+$ and $T_i'^+$ are disjoint implies that the retraction of the $B_i^+$ and the $B_i'^+$ are disjoint. Thus, the $B_i^+$ and $B_i'^+$ are disjoint.

Let $(Q_i',R_i')$ be the boundary edge of $T_i'$, and let $P_i'$ be its midpoint. Since $\gamma{}_i(Q_i',R_i')=(R_i,Q_i)$, it sends the midpoint $P_i'$ to $P_i$. Since $B_i'$ is a connected component in $(\mathbb{P}^1)^{an}\backslash{}\{P_i'\}$, the element $\gamma{}_i$ must send $B_i'$ to a connected component of $(\mathbb{P}^1)^{an}\backslash{}\{P_i\}$. One of the connected components in $(\mathbb{P}^1)^{an}\backslash{}\{P_i\}$ is $(\mathbb{P}^1)^{an}\backslash{}B_i^+$. Since $\gamma{}_i$ sends $Q_i'\in{}B_i'$ to $R_i\in{}(\mathbb{P}^1)^{an}\backslash{}B_i^+$, it must send $B_i'$ to $(\mathbb{P}^1)^{an}\backslash{}B_i^+$. Similarly, $\gamma{}_i^{-1}$ sends $B_i$ to $(\mathbb{P}^1)^{an}\backslash{}B_i'^+$. Thus $F$ is a good fundamental domain in $(\mathbb{P}^1)^{an}$.
\end{proof}

One can establish properties of $S$ similar to Theorem \ref{statement:good_fundamental_domain}. They can be derived either combinatorially or from Proposition \ref{statement:graph_to_analytic_fundamental_domain}.

The following algorithm constructs a good fundamental domain $S$ in $T(\Sigma{})$.

\begin{subroutine}[Good Fundamental Domain Construction]\label{algorithm:construct_good_fundamental_domain}\quad{}
{\rm
\begin{algorithmic}[1]

\Require An ``agent" knowing all vertices and edges of $T=T(\Sigma{})$, and the map $T(\Sigma{})\to{}T(\Sigma{})/\Gamma{}$, where $\Gamma$ is defined over $\Q_p$.
\Ensure A good fundamental domain $S$ in $T(\Sigma{})$.

\State Choose a vertex $P$ of $T$. Let $P_1,\dotsc{},P_k$ be all neighbors of $P$ in $T$.
\State Let $V\gets{}\{P\}$, $E\gets{}\emptyset{}$, $O\gets{}\{(P,P_1),\dotsc{},(P,P_k)\}$, $I\gets{}\emptyset{}$, $A\gets{}\emptyset{}$.
\While {$O\neq{}\phi{}$}
\State Choose $(Q,Q')\in{}O$, remove it from $O$ and add it to $E$.
\State Let $Q,Q_1,\dotsc{},Q_k$ be all neighbors of $Q'$ in $T$.
\State Add $Q'$ to $V$.
\For {each $Q_k$}
\State With the help of the ``agent" in the input, determine if $(Q_k,Q')$ is conjugate to some edge $(R,R')\in{}O$, i.e. $\gamma{}Q_k=R$ and $\gamma{}Q'=R'$ for some $\gamma{}\in{}\Gamma{}$.
\If {so}
\State Remove $(R,R')$ from $O$.
\State Add $(Q',Q_k),(R,R')$ to $I$.
\State Add $\gamma{}$ to $A$.
\Else
\State Add $(Q_k,Q')$ to $O$.
\EndIf
\EndFor
\EndWhile
\State \Return $S=V\cup{}E\cup{}I$. (The edges in $I$ are the boundary edges, and $A$ is a set of free generators of $\Gamma{}$ in good position.)

\end{algorithmic}
}
\end{subroutine}

\begin{proof}
Consider the map from $T(\Sigma{})$ to $G=T(\Sigma{})/\Gamma{}$. Let $P$ be as in Step (1). Suppose that a ``fire'' starts at $P\in{}T(\Sigma{})$ and the image of $P$ in $G$. In each step, when we choose the edge $(Q,Q')$ in Step (4) and add a vertex $Q'$ to $V$ in Step (5), we ``propagate" the fire from $Q$ to $Q'$, and ``burn" $Q'$ together with halves of all edges connecting to $Q'$. Also, we ``burn" the corresponding part in $G$. Suppose two fires meet each other in $G$. In this case, both halves of an edge in $G$ are burned, but it corresponds to two half burned edges in $T(\Sigma{})$. If so, we stop the fire by removing the edges from $O$ and adding them to $I$ (Step (9)). The algorithm terminates when the whole graph $G$ is burned. The burned part $S'$ of $T(\Sigma{})$ is a lifting of $G$. Then, $V$ is the set of vertices of $S'$, $E$ is the set of whole edges in $S'$, and $I$ is the set of half edges in $S'$. The fact that they form a good fundamental domain follows from the method in the proof of \cite[I (4.3)]{GP}.
\end{proof}

This algorithm requires an ``agent" knowing everything about $T(\Sigma{})$. It is hard to construct such an ``agent" because $T(\Sigma)$ is infinite. Therefore, we approximate $T(\Sigma{})$ by a finite subtree. One candidate is $T(\Sigma{}_m)$, where $\Sigma{}_m$ is the set of fixed points of elements of $\Gamma{}_m$. Recall that $\Gamma{}_m$ is the set of elements of $\Gamma{}$ whose reduced words in terms of the given generators have lengths at most $m$. We take one step further: we approximate $T(\Sigma{})$ by $T(\Gamma{}_ma)$, where $a$ is any point in $\Sigma{}$.

\begin{lemma}
For any $a\in{}K$, we have $T(\Gamma{}a)\supset{}T(\Sigma{})$. Furthermore, if $a\in{}\Sigma{}$, then $T(\Gamma{}a)=T(\Sigma{})$.
\end{lemma}

\begin{proof}
For any $g\in{}\Gamma{}$, the fixed point corresponding to the eigenvalue with larger absolute value is the limit of the sequence $a,ga,g^2a,\dotsc{}$. The other fixed point is the limit of the sequence $a,g^{-1}a,g^{-2}a,\dotsc{}$. Therefore, every point in $\Sigma{}$ is either in $\Gamma{}a$ or a limit point of $\Gamma{}a$. Therefore, $T(\Gamma{}a)\supset{}T(\Sigma{})$. The second statement is clear.
\end{proof}

We can construct a complete list of vertices and edges in $T(\Gamma{}_ma)$. Then, the map from $T(\Gamma{}_ma)$ to $T(\Sigma{})/\Gamma{}$ can be approximated in the following way: for each pair of vertices $P,Q$ (resp. edges $e,f$ in $T(\Gamma{}_ma)$ and each given generator $\gamma{}_i$, check if $\gamma{}_iP=Q$ (resp. $\gamma{}_ie=f$). If so, then we identify them. Note that this method may not give the correct map, because two vertices $P$ and $Q$ in $T(\Gamma{}_ma)$ may be conjugate via the action of some $h_1h_2\dotsb{}h_k\in{}\Gamma{}$, where some intermediate step $h_lh_{l+1}\dotsb{}h_kP\notin{}T(\Gamma{}_ma)$. Due to this flaw, we need a way to certify the correctness of the output.

\begin{subroutine}[Good Fundamental Domain Certification]\label{algorithm:certify_good_fundamental_domain}\quad{}
{\rm
\begin{algorithmic}[1]

\Require Generators $\gamma{}_1,\dotsc{},\gamma{}_g\in\mathbb{Q}_p^{2\times 2}$ of a Schottky group $\Gamma{}$, and a quadruple $(V,E,I,A)$, where $V$ is a set of vertices in $T(\Sigma{})$, $E$ and $I$ are sets of edges of $T(\Sigma{})$, $I$ contains $k$ pairs of edges $(P_i,Q_i),(P_i',Q_i')$, where $P_i,P_i'\in{}V$, $Q_i,Q_i'\notin{}V$, and $A$ contains $k$ elements $a_i$ in~$\Gamma{}$.
\Ensure TRUE if $S=V\cup{}E\cup{}I$ is a good fundamental domain in $T(\Sigma{})$ for the set of generators $A$, and $I$ is the set of boundary edges. FALSE otherwise.

\State If $k\neq{}g$, \Return FALSE.
\State If $S$ is not connected, \Return FALSE.
\State If any element of $I$ is not a terminal edge of $S$, \Return FALSE.
\State If any $(P_i,Q_i)\neq{}a_i(Q_i',P_i')$, \Return FALSE.
\State Choose $P$ in the interior of $S$.
\For {$h\in{}\{\gamma{}_1,\dotsc{},\gamma{}_g,\gamma{}_1^{-1},\dotsc{},\gamma{}_g^{-1}\}$}
\State Using a variant of Subroutine \ref{algorithm:reduction_fundamental_domain}, find point $P'\in{}S$ and group element $\gamma{}\in{}\langle{}a_1,\dotsc{},a_k\rangle{}$ such that~$P'=\gamma{}(hP)$.
\State If $P\neq{}P'$, \Return FALSE.
\EndFor
\State \Return TRUE.

\end{algorithmic}
}
\end{subroutine}

\begin{proof}
Steps 1--4 verify that $S$ satisfies the definition of a good fundamental domain in $T(\Sigma{})$ for the set of generators $a_1,\dotsc{},a_g$. In addition, we need to verify that $a_1,\dotsc{},a_g$ generate the same group as the given generators $\gamma{}_1,\dotsc{},\gamma{}_g$. This is done by Steps 5--9. If $P=P'$ in Step 8, then there exists $\gamma{}\in{}\langle{}a_1,\dotsc{},a_k\rangle{}$ such that $\gamma{}hP=P$. We are assuming $a_i\in{}\Gamma{}$ in the input, so $\gamma{}h\in{}\Gamma{}$. Since the action of $\Gamma{}$ on $(\mathbb{P}^1)^{an}\backslash{}\Sigma{}$ is free, we have $\gamma{}h=\mathrm{id}$. Thus, $h\in{}\langle{}a_1,\dotsc{},a_k\rangle{}$. If $P=P'$ for all $h$, then $\Gamma{}=\langle{}a_1,\dotsc{},a_k\rangle{}$.

Otherwise, if $P\neq{}P'$ in Step 8 for some $h$, then there exists $\gamma{}'\in{}\Gamma{}$ such that $\gamma{}'P=P'$. For any $\gamma{}\in{}\langle{}a_1,\dotsc{},a_k\rangle{}$ other than identity, we have $\gamma{}P'\notin{}s^{\circ}$ by a variant of Lemma \ref{statement:location_of_gamma_b}. Therefore, $\Gamma{}\neq{}\langle{}a_1,\dotsc{},a_k\rangle{}$.
\end{proof}

If the certification fails, we choose a larger $m$ and try again, until it succeeds. We are ready to state our main algorithm for this section:

\begin{algorithm}[Turning Arbitrary Generators into Good Generators]\label{algorithm:main_4}\quad{}
{\rm
\begin{algorithmic}[1]

\Require Free generators $\gamma{}_1,\dotsc{},\gamma{}_g\in\mathbb{Q}_p^{2\times 2}$ of a Schottky group $\Gamma{}$.
\Ensure Free generators $a_1,\dotsc{},a_g$ of $\Gamma{}$, together with a good fundamental domain $F=(\mathbb{P}^1)^{an}\backslash{}(B_1\cup{}\dotsb{}\cup{}B_g\cup{}B_1'\cup{}\dotsb{}\cup{}B_g')$ for this set of generators.

\State Let $m=1$.
\State Let $a$ be a fixed point of some $\gamma{}_i$.
\State Compute all elements in $\Gamma{}_ma$.
\State Find all vertices and edges of $T(\Gamma{}_ma)$.
\State Approximate the map $T(\Gamma{}_ma)\to{}T(\Sigma{})/\Gamma{}$.
\State Use Subroutine \ref{algorithm:construct_good_fundamental_domain} to construct a subgraph $S=V\cup{}E\cup{}I$ of $T(\Gamma{}_ma)$ and a subset $A\subset{}\Gamma{}$.
\State Use Subroutine \ref{algorithm:certify_good_fundamental_domain} to determine if $S=V\cup{}E\cup{}I$ is a good fundamental domain in $T(\Sigma{})$.
\State If not, increment $m$ and go back to Step 2.
\State Compute $B_i$ and $B_i'$ from $S$ using the method in Proposition \ref{statement:graph_to_analytic_fundamental_domain}.
\State \Return generators $A$ and good fundamental domain $F=(\mathbb{P}^1)^{an}\backslash{}(B_1\cup{}\dotsb{}\cup{}B_g\cup{}B_1'\cup{}\dotsb{}\cup{}B_g')$.

\end{algorithmic}
}
\end{algorithm}

\begin{proof}
The correctness of the algorithm follows from the proof of Subroutine \ref{algorithm:certify_good_fundamental_domain}. It suffices to prove that the algorithm eventually terminates. Assume that we have the ``agent" in Subroutine \ref{algorithm:construct_good_fundamental_domain}. Since Subroutine \ref{algorithm:construct_good_fundamental_domain} terminates in a finite number of steps, the computation involves only finitely many vertices and edges in $T(\Sigma{})$. If $m$ is sufficiently large, $T(\Sigma{}_m)$ will contain all vertices and edges involved in the computation. Moreover, for any pair of vertices or edges in $T(\Sigma{}_m)$ that are identified in $T(\Sigma{})/\Gamma{}$, there exists a sequence of actions by the given generators of $\Gamma{}$ that sends one of them to the other, so there are finitely many intermediate steps. If we make $m$ even larger so that $T(\Sigma{}_m)$ contains all these intermediate steps, we get the correct approximation of the map $T(\Sigma{})\to{}T(\Sigma{})/\Gamma{}$. This data is indistinguishable from the ``agent" in the computation of Subroutine \ref{algorithm:construct_good_fundamental_domain}. Thus, it will output the correct good fundamental domain.
\end{proof}

\begin{remark}
The performance of the algorithm depends on how ``far" the given generator is from a set of generators in good position, measured by the lengths of the reduced words of the good generators in terms of the given generators. If the given generators is close to a set of generators in good position, then a relatively small $m$ is sufficient for $T(\Gamma{}_ma)$ to contain all relevant vertices. Otherwise, a larger $m$ is needed. For example, in the genus $2$ case, this algorithm terminates in a few minutes for our test cases where each given generator has a reduced word of length $\leq{}4$ in a set of good generators. However, the algorithm is not efficient on Example \ref{example:period_matrix} (4), where one of the given generators has a reduced word of length $101$. One possible way of speeding up the algorithm is to run the non-Euclidean Euclidean algorithm developed by Gilman \cite{Gi} on the given generators.
\end{remark}

\begin{remark}\label{remark:schottky_test}
We may relax the requirement that the input matrices freely generate a Schottky group by checking that every element in $\Gamma{}_m$ not coming from the empty word is hyperbolic before Step 3. If the group is Schottky and freely generated by the input matrices, the algorithm will terminate with a good fundamental domain. Otherwise, Step 7 will never certify a correct good fundamental domain, but the hyperbolic test will eventually fail when a non-hyperbolic matrix is generated. In particular, if the identity matrix is generated by a nonempty word, the generators are not free (though they may or may not generate a Schottky group); and if a non-identity hyperbolic matrix is generated, the group is not Schottky.  Thus, Algorithm \ref{algorithm:main_4} is turned into a Schottky test algorithm. Again, the non-Euclidean Euclidean algorithm in \cite{Gi} is a possible ingredient for a more efficient Schottky test algorithm.
\end{remark}

\section{Future Directions:  Reverse Algorithms and Whittaker Groups}\label{futurequestionssection}

In this section we describe further computational questions about Mumford curves.  Algorithms answering these questions would be highly desirable.

\subsection{Reversing The Algorithms in Section \ref{section:main_algorithms}}  Many of our main algorithms answer questions of the form ``Given $A$, find $B$'', which we can reverse to ``Given B, find $A$.''  For instance:

\begin{itemize}

\item  Given a period matrix $Q$, determine if the abelian variety $(K^*)^g/Q$ is the Jacobian of a Mumford curve, and if it is approximate the corresponding Schottky group.

\item  Given an abstract tropical curve $G$, find a Schottky group whose Mumford curve has $G$ as its abstract tropical curve.

\item  Given a polynomial representation of a curve, determine if it is a Mumford curve, and if it is approximate the corresponding Schottky group.
\end{itemize}

A particular subclass of Schottky groups called \emph{Whittaker groups} are likely a good starting point for these questions.  

\subsection{Whittaker Groups}
\label{sectionwhittaker} We will outline the construction of Whittaker groups (see \cite{Pu1} for more details), and discuss possible algorithms for handling computations with them.  We are particularly interested in going from a matrix representation to a polynomial representation, and vice versa.

 If $s\in PGL(2,K)$ is an element of order $2$, then $s$ will have two fixed points, $a$ and $b$, and is in fact determined by the pair $\{a,b\}$ as $$s=\left[\begin{matrix}a&b\\1&1\end{matrix}\right]\left[\begin{matrix}1&0\\0&-1\end{matrix}\right]\left[\begin{matrix}a&b\\1&1\end{matrix}\right]^{-1},$$
 as long as $\infty\neq a,b$.
Let $s_0,\ldots,s_g$ be $g+1$ elements of $PGL(2,K)$ of order $2$.  Write their fixed points as $\{a_0,b_0\}, \ldots, \{a_g,b_g\}$, and assume without loss of generality that $\infty\neq a_i,b_i$ for all $i$.  Let $B_0,\ldots, B_g$ denote the smallest open balls containing each pair, and assume that the corresponding closed balls $B_0^+,\ldots ,B_g^+$ are all disjoint.  
Then the group $\Gamma:=\left<s_0,\ldots,s_g\right>$ is in fact the free product $\left<s_0\right>*\cdots*\left<s_g\right>$. 

Note that $\Gamma$ is \emph{not} a Schottky group, since its generators are not hyperbolic.  However, we can still consider its action upon $\pro^1\setminus\Sigma=\Omega$.  To fix some notation, we will choose $a,b\in\Omega$ such that $a\notin \Gamma b$ and $\infty\notin\Gamma a\cup\Gamma b$ and will define
$$G(z):=\Theta(a,b;z)=\prod_{\gamma\in \Gamma}\frac{z-\gamma(a)}{z-\gamma(b)}.$$
 (In our previous definition of theta functions, we took $\Gamma$ to be Schottky, but the definition works fine for this $\Gamma$ as well.)  If we choose $a$ and $b$ such that $|G(\infty)-G(s_0\infty)|<1/2$, then $G$ will be invariant under $\Gamma$, which gives a morphism $\Omega/\Gamma\rightarrow \pro^1$.  This will have only one pole, so it is an isomorphism.

Now, let $W$ be the kernel of the map $\varphi:\Gamma\rightarrow\Z/2\Z$ defined by $\varphi(s_i)=1$ for all $i$.  Then $W=\left<s_0s_1,s_0s_2,\ldots,s_0s_g\right>$, and is in fact free on those generators.  One can show that $W$ is a Schottky group of rank $g$, and we call a group that arises in this way a \emph{Whittaker group}.  We already know that $\Omega/W$ is a curve of genus $g$; in fact, we have more than that.

\begin{theorem}[Van der Put, \cite{Pu1}] If $W$ is a Whittaker group, then $\Omega/W$ is a totally split hyperelliptic curve of genus $g$, with affine equation $y^2=\prod_{i=0}^g(x-G(a_i))(x-G(b_i))$.  Conversely, if  $X$ be a totally split hyperelliptic curve of genus $g$ over $K$, then there exists a Whittaker group $W$ such that $X\cong \Omega/W$, and this $W$ is unique up to conjugation in $PGL(2,K)$.
\end{theorem}

\begin{remark}  There is a natural map $\Omega/W\rightarrow \Omega/\Gamma\cong \pro^1$.  This is the expected morphism of degree $2$ from the hyperelliptic curve to projective space, ramified at $2g+2$ points.
\end{remark}

If we are content with an algorithm taking $s_0,s_1,\ldots,s_g$ as the input representing a Whittaker group $W$ (so that $W=\left<s_0s_1,\ldots,s_0s_g\right>$), the above theorem tells us how to compute the ramification points of the hyperelliptic Mumford curve $\Omega/W$.

\begin{example}  Let's construct an example of a Whittaker group of genus $2$ with $K=\mathbb{Q}_3$.  We need to come up with matrices $s_0, s_1, s_2$ of order $2$ with fixed points sitting inside open balls whose corresponding closed balls are disjoint.  We will choose them so that the fixed points of $s_0$ are $0$ and $9$; of $s_1$ are $1$ and $10$; and of $s_2$ are $2$ and $11$.  (The smallest \emph{open } balls containing each pair of points has radius $\frac{1}{3}$, and the corresponding \emph{closed} balls of radius $\frac{1}{3}$ are disjoint.)  The eigenvalues will be $1$ and $-1$, and the eigenvectors are the fixed points (written projectively), so we can take
$$s_0=\left[\begin{matrix}0&9\\1&1\end{matrix}\right]^{-1}\left[\begin{matrix}1&0\\0&-1\end{matrix}\right]  \left[\begin{matrix}0&9\\1&1\end{matrix}\right]=
\left[\begin{matrix}-1&0\\-\frac{2}{9}&1\end{matrix}\right]  $$
$$s_1=\left[\begin{matrix}1&10\\1&1\end{matrix}\right]^{-1}\left[\begin{matrix}1&0\\0&-1\end{matrix}\right]  \left[\begin{matrix}1&10\\1&1\end{matrix}\right]=
\left[\begin{matrix}-\frac{11}{9}&\frac{20}{9}\\-\frac{2}{9}&\frac{11}{9}\end{matrix}\right]  $$
$$s_2=\left[\begin{matrix}2&11\\1&1\end{matrix}\right]^{-1}\left[\begin{matrix}1&0\\0&-1\end{matrix}\right]  \left[\begin{matrix}2&11\\1&1\end{matrix}\right]=
\left[\begin{matrix}-\frac{13}{9}&\frac{44}{9}\\-\frac{2}{9}&\frac{13}{9}\end{matrix}\right].  $$

So the group $$\Gamma=\left<
\left[\begin{matrix}-1&0\\-\frac{2}{9}&1\end{matrix}\right] ,
\left[\begin{matrix}-\frac{11}{9}&\frac{20}{9}\\-\frac{2}{9}&\frac{11}{9}\end{matrix}\right]  ,
\left[\begin{matrix}-\frac{13}{9}&\frac{44}{9}\\-\frac{2}{9}&\frac{13}{9}\end{matrix}\right] \right>$$
is  generated by those three elements of order $2$ (and is in fact the free product of the groups $\left<s_0\right>$, $\left<s_1\right>$, and $\left<s_2\right>$), and its subgroup
$$W=\left<s_0s_1,s_0s_2\right>=\left<
\left[\begin{matrix}\frac{59}{81}&\frac{20}{9}\\-\frac{4}{81}&\frac{11}{9}\end{matrix}\right], 
\left[\begin{matrix}\frac{29}{81}&\frac{44}{9}\\-\frac{8}{81}&\frac{13}{9}\end{matrix}\right]\right> $$
is a Whittaker group of rank $2$. 


 The quotient $\Omega/W$ is a hyperelliptic curve of genus $2$, with six points of ramification $G(0), G(1), \ldots,G(5)$, where $G$ is the theta function for $\Gamma$ with suitably chosen $a$ and $b$.

\end{example}

\begin{question}\rm{As long as we know the $2$-torsion matrices $s_0,s_1,\ldots,s_g$ that go into making a Whittaker group, we can find the ramification points of the corresponding hyperelliptic curve.  But what if we don't have that data?

  \bi

\item If we are given $W=\left<\gamma_1,\ldots,\gamma_g\right>$, can we algorithmically determine whether or not $W$ is Whittaker? 

\item  If we know $W=\left<\gamma_1,\ldots,\gamma_g\right>$, can we algorithmically find  $s_0,s_1,\ldots,s_g$ from $\gamma_1,\ldots,\gamma_g$?  

\item If we know $W$ is Whittaker but cannot find $s_0,s_1,\ldots,s_g$, is there another way to find the ramification points of  $\Omega/W$?  

\ei
}
\end{question}

A good first family of examples to consider is Schottky groups generated by two elements.  These give rise to genus $2$ curves, which are hyperelliptic, so the groups must in fact be Whittaker.

Having discussed going from a Whittaker group to a set of ramification points, we now consider the other direction:  going from the ramification points of a totally split hyperelliptic curve and finding the corresponding Whittaker group.  This more difficult, though a brute force method was described by  Kadziela in \cite{Ka}, and was used to compute several genus $2$ examples over $\mathbb{Q}_5$.  We will outline his approach.

After a projective transformation, we may assume that the set of fixed points of the group $\Gamma$ is of the form
$$S=\{0,b_0, a_1,b_1,\ldots a_{g-1},b_{g-1}, 1,\infty\}, $$
where
$$0<|b_0|<|a_1|\leq \ldots \leq |b_{g-1}|<1,$$
and where the generators of $\Gamma$ are the $2$-torsion matrices $s_i$ with fixed points $\{a_i,b_i\}$ (taking $a_0=0$, $a_g=1$, and $b_{g}=\infty$).  Let us choose parameters for the theta function associated to $\Gamma$ as $0$ and $1$, and write
$$G(z)=\Theta(0,1;z)=\prod_{n=0}^\infty L_n(z), $$
where
$$L_n(z):=\prod_{\gamma\in \Gamma,\ell(\gamma)=n}\frac{z-\gamma(0)}{z-\gamma(1)} $$
is the sub product of $\Theta$ over all matrices in $\Gamma$ with reduced length exactly $n$.

\begin{theorem}[Kadziela's Main Approximation Theorem, \cite{Ka}]\label{theorem:kadziela}  Assume $S$ and $G$ are as above, and let $\pi$ denote the uniformizer.  Then
$$G(0)=0,G(1)=\infty, G(\infty)=1, $$
and for $z\in S-\{0,1,\infty\}$,
\bi

\item  $G(z)\equiv 0\mod \pi$

\item  $G(z)\equiv\begin{cases}
-4b_0\mod \pi^2 \text{  if $z=b_0$}, \\
-2z\mod\pi^2 \text{  if $z\neq b_0$}
\end{cases}$

\item  $G(z)\mod \pi^t-\prod_{i=0}^{t-2}L_i(z)\mod \pi^t=\prod_{i=0}^{t-2} L_i(z\mod\pi^t)$ for $t\geq 3$.

\ei
\end{theorem}

Let $X$ be a totally split hyperelliptic curve of genus $g$, which after projective transformation we may assume has its set of ramification points in the form
$$R=\{0,r_0,\ldots,r_{2g-2},1,\infty\} $$
where $0<|r_0|<|r_1|\leq\ldotsÉ\leq |r_{2g-2}|<1$.  We know $X\cong \Omega/W$ for some Whittaker group $W$.  To find $W$ it will suffice to find the fixed points $S$ of the corresponding group $\Gamma$, so given $R$ we wish to find $S$.  We know $S=\Theta^{-1}(R)$, but $\Theta$ is defined by $S$, and we cannot immediately invert a function we do not yet know.  This means we must gradually approximate candidates for both $S$ and $\Theta$ that give the desired property that $\Theta(S)=R$.  To simplify notation, we will sometimes write $S=\{0,x_0,x_1,\ldots,x_{2g-2},1,\infty\}$ instead of in terms of $a_i$'s and $b_i$'s.  

The following algorithm follows the description in \cite[\textsection 6]{Ka}.  Although we have not implemented it, Kadziela used a Magma implementation of it to compute several genus $2$ examples over $\mathbb{Q}_5$.

\begin{algorithm}[From Ramification Points to Whittaker Group]\label{algorithm:whittaker_group}\quad{}
{\rm
\begin{algorithmic}[1]
\Require Set of ramification points $R=\{0,r_0,\ldots,r_{2g-2},1,\infty\}\subset\Q_p\cup\{\infty\}$, and desired degree of precision $d\geq 3$  
\Ensure The set of fixed points $S=\{x_0,\ldots,x_{2g-2}, 1,\infty\}$ of $\Gamma$, approximated $\mod \pi^d$, such that $\Omega/W$ has ramification points $R$ for the corresponding Whittaker group of $\Gamma$.
\State Sort  $r_0,\ldots,r_{2g-1}$ in increasing absolute value and rename.
\If {$|r_0|=|r_1|$}
\State \Return ``NOT VALID''
\EndIf
\State Define $x_i=0$ for $0\leq i\leq 2g-2$.  (Approximation $\mod\pi$.)
\State Let $m=\max\{k\,|\,r_k\equiv 0\mod \pi^2\}$.
\State Set $\ell=0$ and GOOD=FALSE
\While {GOOD=FALSE}
\State Set $x_0=-\frac{1}{4}r_i\mod\pi^2$, and all other $x_j$'s to the $-\frac{1}{2}r_k\mod\pi^2$.
\State Test if $i$ is the right choice using Theorem \ref{theorem:kadziela}; if it is, set GOOD=TRUE
\State Set $\ell=\ell+1$.
\EndWhile
\For {$3\leq t\leq d$}
\State Set DONE=FALSE.
\While {DONE=FALSE}
\State Choose $\overline{v}\in (\mco_K/\mfm\mco_K)^{2g-1}$, set $\overline{x}=\left(\overline{x}\mod\pi^{t-1}\right)+\overline{v}\pi^t$.
\State Compute $\prod_{n=0}^{t-2}L_n(x_i)$ for $0\leq i\leq 2g-2$.
\If {this set equals $\{r_0\mod\pi^t,\ldots r_{2g-1}\mod\pi^t\}$}
\State Set DONE=TRUE.
\Else
\State Set DONE=FALSE.
\EndIf
\EndWhile
\EndFor
\State \Return $x_0,\ldots x_{2g-1}$.
\end{algorithmic}
}
\end{algorithm}

This algorithm is in some sense a brute force algorithm,  as for each digit's place from $3^{rd}$ to $d^{th}$ it might in principal try every element of $(\mco_K/\mfm\mco_K)^{2g-1}$, essentially guessing the $x_i$'s digit by digit (lines 13 through 24). It is nontrivial that such a brute force method could even work, but this is made possible by Theorem \ref{theorem:kadziela} as it tells us how to check whether a choice of element in $(\mco_K/\mfm\mco_K)^{2g-1}$ is valid $\mod\pi^m$.  As with the other algorithms presented in this paper, future algorithms improving the efficiency would be greatly desirable.

\end{document}